\documentclass[a4paper,11pt]{amsart}
\usepackage{amsfonts,amssymb,amsthm,cite,amsmath,amstext}
\usepackage{color}
\usepackage[colorlinks,linkcolor=black,citecolor=black]{hyperref}
\usepackage{comment}
\usepackage{framed}
\usepackage{calligra} 

\definecolor{shadecolor}{gray}{0.875}

\newtheorem{thrm}{Theorem}[section]
\newtheorem{thrmx}{Theorem}

\newtheorem{corx}{Corollary}

\newtheorem{lem}[thrm]{Lemma}
\newtheorem{cor}[thrm]{Corollary}
\newtheorem{prop}[thrm]{Proposition}

\theoremstyle{definition}
\newtheorem{defn}[thrm]{Definition}

\newtheorem{exmple}[thrm]{Example}
\newtheorem{rmk}[thrm]{Remark}

\newtheorem{ques}[thrm]{Question}

\DeclareMathOperator{\Gr}{Gr}

\DeclareMathOperator{\nd}{nd}

\DeclareMathOperator{\Span}{span}

\DeclareMathOperator{\codim}{codim}
\DeclareMathOperator{\MSupp}{MSupp}
\DeclareMathOperator{\HL}{HL}
\DeclareMathOperator{\HR}{HR}
\DeclareMathOperator{\Herm}{Herm}
\DeclareMathOperator{\trdeg}{trdeg}

\title{Hard Lefschetz properties, complete intersections and numerical dimensions}
\author{Jiajun Hu and Jian Xiao}
\date{}

\begin{document}
\maketitle


\begin{abstract}
We study the positivity of complete intersections of nef classes. We first give a sufficient and necessary characterization on the complete intersection classes which have hard Lefschetz property on a compact complex torus, equivalently, in the linear case. In turn, this provides us new kinds of cohomology classes which have Hodge-Riemann property or hard Lefschetz property on an arbitrary compact K\"ahler manifold. We also give a complete characterization on when the complete intersection classes are non-vanishing on an arbitrary compact K\"ahler manifold. Both characterizations are given by the numerical dimensions of various partial summations of the given nef classes. As an interesting byproduct, we show that the numerical dimension endows any finite set of nef classes with a loopless polymatroid structure.
\end{abstract}

\tableofcontents

\section{Introduction}

In this paper, we work over the field of complex numbers $\mathbb{C}$, while corresponding results will be mentioned when they also hold over an arbitrary algebraically closed field $k$.


As summarized in \cite{huhICM2018}, given a mathematical object $X$ of ``dimension'' $d$, it is important to study whether a K\"ahler package holds on $X$. More precisely, sometimes one can construct from $X$ in a natural way a real or complex graded vector space (or even a graded algebra)
$$A(X) = \bigoplus_{q=0} ^d A^q (X),$$
together with a bilinear map $P: A(X)\times A(X) \rightarrow \mathbb{C}$ and a graded linear map $L: A^{\bullet} (X) \rightarrow A^{\bullet+1} (X)$ that is symmetric with respect to $P$. The linear operator $L$ usually comes from a convex cone $K(X)$ in the space of linear operators on $A(X)$.
For a nonnegative integer $q\leq d/2$, a \emph{K\"ahler package} on $X$ usually means the following properties:
\begin{description}
  \item[(PD)] The bilinear paring
  \begin{equation*}
    A^q (X) \times A^{d-q} (X) \rightarrow \mathbb{C},\ (\xi, \eta)\mapsto P(\xi, \eta)
  \end{equation*}
  in non-degenerate (the Poincar\'{e} duality for $X$)
  \item[(HL)] For any $L_1, ..., L_{d-2q}\in K(X)$, the linear map
  \begin{equation*}
    A^q (X) \rightarrow A^{d-q} (X),\ \xi \mapsto \left(\prod_{k=1} ^{d-2q} L_k \right) \cdot \xi
  \end{equation*}
  is bijective (the hard Lefschetz theorem for $X$).
  \item[(HR)] For any $L_1, ..., L_{d-2q+1}\in K(X)$, the quadratic form
  \begin{equation*}
    A^q (X) \times A^q (X) \rightarrow \mathbb{C},\ (\xi_1, \xi_2)\mapsto (-1)^q P\left(\xi_1, \left(\prod_{k=1} ^{d-2q} L_k \right)\overline{\xi_2}\right)
  \end{equation*}
  is positive definite on the kernel of the linear map
  \begin{equation*}
    A^q (X) \rightarrow A^{d-q +1} (X), \ \xi \mapsto \left(\prod_{k=1} ^{d-2q+1} L_k \right) \cdot \xi
  \end{equation*}
  (the Hodge-Riemann relation for $X$).
\end{description}
Sometimes, the set $K(X)$ is a convex cone of $A^1 (X)$, which is called the K\"ahler cone of $X$. In last decades, novel and exciting K\"ahler packages were discovered and have been playing key roles in algebra, combinatorics and geometry, see e.g. \cite{huhICM2018, huh2022combinatorics, williamECM} and the references therein.

In this paper, we are interested in the HL and HR parts, and study for which kinds of $L_1,...,L_{d-2q}$ the HL or HR property will hold. We focus on the model and original case -- the K\"ahler package on a compact K\"ahler manifold.

Let $X$ be a compact K\"ahler manifold of dimension $n$ and $\omega$ a K\"ahler class on $X$. As a fundamental piece of Hodge theory, for any integers $0\leq p, q \leq p+q \leq n$, the complete intersection class $$\Omega = \omega^{n-p-q} \in H^{n-p-q, n-p-q} (X, \mathbb{R}),$$
has the following properties:
\begin{description}
  \item [(HL)] The linear map $$\Omega: H^{p,q}(X, \mathbb{C})\rightarrow H^{n-q,n-p}(X, \mathbb{C}),\ \phi \mapsto \Omega \cdot \phi$$
      is an isomorphism.
  \item [(HR)] The quadratic form $Q$ on $H^{p, q}(X, \mathbb{C})$, defined by
  \begin{equation*}
    Q(\varphi_1, \varphi_2) = c_{p, q} \Omega \cdot  \varphi_1 \cdot \overline{\varphi_2}, \ \text{where}\ c_{p, q} = \mathrm{i}^{q-p} (-1)^{(p+q)(p+q+1)/2},
  \end{equation*}
  is positive definite on the primitive space $P^{p, q}(X, \mathbb{C})$ with respect to $(\Omega, \omega)$, i.e., $\phi \in P^{p, q}(X, \mathbb{C})$ if and only if $\Omega \cdot \omega\cdot \phi=0$.
\end{description}

Note that we arrive at a special case of the above abstract setting for $A(X)$ when $p=q$.

We call that \emph{the class $\Omega$ has HL property} and \emph{the pair $(\Omega, \omega)$ has HR property}.


By replacing the above $\Omega=\omega^{n-p-q}$ by an arbitrary cohomology class $\Omega \in H^{n-p-q, n-p-q} (X, \mathbb{R})$ and $\omega$ by an arbitrary $(1,1)$ class $\eta \in  H^{1, 1} (X, \mathbb{R})$,
it is interesting to study:
\begin{quote}
  When does the class $\Omega$ have HL property and when does the pair $(\Omega, \eta)$ have HR property? (In order to define the primitive space we need the class $\Omega$ coupling with an $(1,1)$ class $\eta$.)
\end{quote}

A complete answer seems unreachable at this moment, nevertheless, we first study a subclass of $\Omega$, coming from complete intersections:

\begin{ques}\label{ques HL}
Let $X$ be a compact K\"ahler manifold of dimension $n$, and let $\alpha_1,....,\alpha_{n-p-q}, \eta \in H^{1,1}(X, \mathbb{R})$, then under which assumptions does the complete intersection class $$\Omega=\alpha_1 \cdot.... \cdot\alpha_{n-p-q}$$ have HL property and does the pair $(\Omega, \eta)$ have HR property?
\end{ques}

\subsubsection*{Known results}
In \cite{DN06} (see also \cite{timorinMixedHRR, gromov1990convex} for related results) and \cite{cattanimixedHRR}, it was proved that given any K\"ahler classes $\omega_1, ..., \omega_{n-p-q+1}$,
the complete intersection class $$\Omega=\omega_1 \cdot ... \cdot \omega_{n-p-q}$$ has HL property,
and the pair $(\Omega, \omega_{n-p-q+1})$ has HR property.
This greatly generalizes the classical K\"ahler package to the mixed situation, in the sense that $\Omega$ can be a complete intersection of different K\"ahler classes. This result has important applications in high dimensional complex dynamics.

In the previous works of the second author \cite{xiaoHodgeIndex, xiaomixedHRR} (see also \cite{zhangHRR} for further refinement), inspired by Hodge index theorems deduced from complex Hessian equations and hyperbolic polynomials, it was realized that the classical positivity assumption -- being K\"ahler or ample -- can be weakened to still ensure that $\Omega$ has HL and HR properties.

More precisely, let $\omega$ be a K\"ahler class and $\alpha_1, ..., \alpha_{m-1}$ be $m$ positive $(1,1)$ classes -- whose definition will be recalled in Section \ref{sec pre}, then by \cite{xiaoHodgeIndex}, for $(p,q)=(1,1)$,
the complete intersection class $$\Omega = \omega^{n-m}\cdot \alpha_1\cdot ...\cdot \alpha_{m-2}$$
has HL property, and the pair $(\Omega, \alpha_{m-1})$ has HR property.
Note that an $m$ positive $(1,1)$ class ($m<n$) can be degenerate and even negative in certain directions, which is a positivity notion much weaker than being K\"ahler or ample. For general $(p,q)$, let $p, q, m$ be integers such that $0\leq p, q\leq p+q \leq m \leq n$. Let $\omega_1, ...,\omega_{n-m}$ be K\"ahler classes, and let $\alpha_1, ...,\alpha_{m-p-q+1}$ be semi-positive and $m$ positive, then by \cite[Theorem A and Remark 3.9]{xiaomixedHRR} we have
that the complete intersection class $$\Omega = \omega_1 \cdot ... \cdot \omega_{n-m}\cdot \alpha_1\cdot ...\cdot \alpha_{m-p-q}$$
has HL property, and that the pair $(\Omega, \alpha_{m-p-q+1})$ has HR property.
These results pushforward \cite{DN06} and \cite{cattanimixedHRR} to the mixed setting allowing degenerate positivity.

Still with degenerate positivity, in a more geometric setting \cite{decataldoLefsemismall} proved the following hard Lefschetz property: if $L$ is a line bundle on a complex projective manifold $X$ such that a positive power of $L$ is generated by its global sections, then the class $\Omega=c_1 (L) ^{n-p-q}$
has HL property if and only if $L$ is lef, in the sense that
$mL=f^*A$
for some projective semismall morphism $f: X\rightarrow Y$, where $A$ is an ample line bundle on $Y$. In particular, the pair $(\Omega, c_1 (L))$ has HR property. The map $f: X\rightarrow Y$ is called semismall if $\dim Y^k + 2k \leq \dim X$, where $Y^k = \{y\in Y | \dim f^{-1} (y) =k\}$. This has been playing important roles in the study of Hodge theory of algebraic maps.

More recently, it was proved in a series of works \cite{ross2019hodge, rossHR, rossHRkahler} that characteristic forms/classes of ample vector bundles also produce K\"ahler packages. Let $\omega$ be a K\"ahler class and $E$ an ample vector bundle on $X$, then for $(p,q)=(1,1)$,
the Schur class $$\Omega=s_\lambda (E), \ \text{with}\ |\lambda|=n-2,$$
has HL property and the pair $(\Omega, \omega)$ has HR property.
Here we only mention a prototype of their results, for more general statements, we refer the reader to the aforementioned references.

\subsubsection*{Our results}

When all the $(1,1)$ classes $\alpha_k$ are nef, we give a complete characterization on those $\Omega$ which have HL property, by using the positivity properties of various partial summations of $\alpha_1,...,\alpha_{n-p-q}$, on a compact complex torus, equivalently, in the linear case.

To present our result, we first introduce some notations.

Given a nef $(1,1)$ class $\alpha$ on a compact K\"ahler manifold $X$, its numerical dimension, denoted by $\nd(\alpha)$, is defined by
\begin{equation*}
  \nd(\alpha) = \max\{k\in \mathbb{N}| \alpha^k \neq 0\}.
\end{equation*}
Equivalently, fix a K\"ahler class $\omega$, $\nd(\alpha)$ is the largest integer $k$ such that $\alpha^k \cdot \omega^{n-k} > 0$.

We denote the set of indices $\{1,2,...,m\}$ by $[m]$. Given finite vectors $v_1, ...,v_m$ in a vector space $V$, for $I\subset [m]$ write
$$v_I = \sum_{i\in I} v_i. $$
For a subset $I\subset [m]$, we denote the cardinality of $I$ by $|I|$.

\begin{thrmx} \label{hardlef tori}
Let $X=\mathbb{C}^n / \Gamma$ be a compact complex torus of dimension $n$ and $0\leq p,q\leq p+q\leq n$. Let $\alpha_1,...,\alpha_{n-p-q}, \eta\in H^{1,1} (X, \mathbb{R})$ be nef classes on $X$. Denote $\Omega=\alpha_1\cdot...\cdot\alpha_{n-p-q}$. Then for HL property, the following statements are equivalent:
\begin{enumerate}
  \item the intersection class $\Omega$ has HL property;
  \item for any subset $I \subset [n-p-q]$, $\nd(\alpha_I)\geq |I|+p+q$.
\end{enumerate}
For HR property, the following statements are equivalent:
\begin{enumerate}
  \item the pair $(\Omega, \eta)$ has HR property for $\eta$ with $\nd(\eta)\geq p+q$;
  \item for any subset $I \subset [n-p-q]$, $\nd(\alpha_I)\geq |I|+p+q$.
\end{enumerate}

\end{thrmx}

Theorem \ref{hardlef tori} is closely related to Alexandrov's seminal work \cite{alexandroff1938theorie} and Panov's further generalization \cite{Panov1987ONSP}, see Remark \ref{panov summary}. Indeed, Alexandrov's and Panov's related results are special cases of Theorem \ref{hardlef tori} when $p=q=1$, $p=q=0$.

As a consequence of Theorem \ref{hardlef tori}, we obtain new kinds of cohomology classes on an arbitrary compact K\"ahler manifold, which have HL and HR properties.
The following result generalizes \cite{xiaomixedHRR}.

\begin{corx}\label{cor kahler hrr intro}
Let $X$ be a compact K\"ahler manifold of dimension $n$ and $0\leq p,q\leq p+q\leq n$. Let $\alpha_1,...,\alpha_{n-p-q}, \eta \in H^{1,1}(X,\mathbb{R})$ be nef classes on $X$. Denote $\Omega=\alpha_1\cdot...\cdot\alpha_{n-p-q}$.
Assume that there exists a smooth semi-positive representative $\widehat{\alpha}_i$ in each class $\alpha_i$, such that for any subset $I \subset [n-p-q]$, $\widehat{\alpha}_I$ is $|I|+p+q$ positive in the sense of forms, and that there exists a smooth semi-positive representative $\widehat{\eta}$ in $\eta$, such that $\widehat{\eta}$ is $p+q$ positive, then
\begin{itemize}
  \item the complete intersection class $\Omega$ has HL property;
  \item the pair $(\Omega, \eta)$ has HR property.
\end{itemize}
\end{corx}

\begin{rmk}
Similar to Theorem \ref{hardlef tori}, we expect a complete characterization of HL property for complete intersections of nef classes on an arbitrary compact K\"ahler manifold. The requirement for numerical dimensions of various partial summations should be necessary. However, it is clear that this is not sufficient. The sufficiency is more subtle, we hope to address this in a future work.
\end{rmk}

We consider HL and HR properties as positivity properties of $\Omega$.

Besides HL or HR property, we also study a weaker positivity notion, that is, the non-vanishing property of the complete intersection class $$\Omega=\alpha_1\cdot...\cdot\alpha_m$$
when each $\alpha_k\in H^{1,1} (X, \mathbb{R})$ is a nef class.

Applying Theorem \ref{hardlef tori} to the case $p=q=0$, for nef classes $\alpha_1,...,\alpha_{n-p-q}, \eta\in H^{1,1} (X, \mathbb{R})$ on a compact complex torus the following are equivalent:
\begin{enumerate}
  \item $\alpha_1\cdot...\cdot\alpha_n >0$;
  \item for any subset $I \subset [n]$, $\nd(\alpha_I)\geq |I|$.
\end{enumerate}

It is natural to expect that this also holds on an arbitrary compact K\"ahler manifold.
Indeed, we obtain the following complete characterization.

\begin{thrmx}\label{intro thrm interPosi}
Let $X$ be a compact K\"ahler manifold of dimension $n$ and $1\leq m\leq n$. Let $\alpha_1,...,\alpha_m\in H^{1,1} (X, \mathbb{R})$ be nef classes on $X$. Denote $\Omega=\alpha_1\cdot...\cdot\alpha_m$, then the following statements are equivalent:
\begin{enumerate}
  \item the intersection class $\Omega \neq 0 $ in $H^{m,m} (X, \mathbb{R})$;
  \item for any subset $I \subset [m]$, $\nd(\alpha_I) \geq |I|$.
\end{enumerate}

\end{thrmx}

Similar result also holds for compact K\"ahler varieties (see Corollary \ref{kahler variety interspos}).
In the K\"ahler setting, the above $\Omega \neq 0$ is equivalent to that $\Omega$ has a representative which is a non-zero positive current of bidegree $(m,m)$.

\begin{rmk}
The characterization in Theorem \ref{intro thrm interPosi} structurally looks similar to the characterization of hard Lefschetz property. Indeed, our proof applies a Hodge index type theorem,
and the same criterion holds on a smooth projective variety over an arbitrary algebraically closed field.
\end{rmk}

Theorem \ref{intro thrm interPosi} is analogous to the convexity result
\cite[Theorem 5.1.8]{schneiderBrunnMbook}. Given convex bodies $K_1, ..., K_n$ in $\mathbb{R}^n$, the following statements are equivalent:
\begin{itemize}
  \item The mixed volume $V(K_1, ...,K_n)>0$;
  \item There are segments $S_i \subset K_i$, $1\leq i\leq n$, with linearly independent directions;
  \item $\dim K_I \geq |I|$ for any $I\subset [n]$.
\end{itemize}

Among the applications of Theorem \ref{intro thrm interPosi}, we use it to study when the multidegrees are positive in the K\"ahler setting (see Theorem \ref{multidegree}), extending related results in \cite{CCLMZ20}, where multidegrees for algebraic varieties and their connection with algebra and combinatorics were systematically explored.

In the algebraic setting, the characterization of multidegrees is used to show that the support of multidegrees is a discrete polymatroid, described as the integer points in a generalized permutohedron. These results are closely related to submodular functions.

We relate the characterization of positivity of multidegrees of a K\"ahler submanifold to numerical dimensions. Combining with the submodularity property in the algebraic setting, we are motivated to ask whether the numerical dimension of nef classes is a submodular function.

We prove the following general result.

\begin{thrmx}\label{polymatroid}
Let $X$ be a complex projective manifold of dimension $n$, then for any three nef classes $A,B,C \in H^{1,1}(X, \mathbb{R})$ on $X$, we always have
\begin{equation*}
  \nd(A+B+C) + \nd(C)\leq \nd(A+C)+\nd(B+C).
\end{equation*}

As a consequence, for any finite set of nef classes $E=\{B_1, ..., B_m\}$ on $X$, for $I\subset [m]$ set $$r(I)=\nd(B_I)$$ with the convention that $r(\emptyset)=0$, then the function $r(\cdot)$ endows with $E$ a loopless polymatroid structure.

\end{thrmx}

\begin{rmk}
The analogous result also holds on a smooth projective variety over an arbitrary algebraically closed field. We expect that Theorem \ref{polymatroid} holds on an arbitrary compact K\"ahler manifold, and establish a special case for semi-positive classes (see Proposition \ref{rank on Kahler}).
\end{rmk}

We apply this kind of results to show that discrete polymatroids appear naturally in the study of positivity of multidegrees in the K\"ahler setting and in the study of HL property in the linear setting (see Corollary \ref{support}, Proposition \ref{polym HL linear}).

\subsection*{Organization} This paper is organized as follows. We present some basic definitions and notations in Section \ref{sec pre}. In Section \ref{sec hardlef}, we give a sufficient and necessary characterization on complete intersection classes of nef classes to have HL property on a compact complex torus, and discuss its some applications. In Section \ref{sec nonvanishing}, we give a complete characterization on the non-vanishing of complete intersections of nef classes on a compact K\"ahler manifold. Finally, in Section \ref{sec polymatroid} we discuss various combinatorial structures closely related to the previous sections.

\subsection*{Acknowledgements}
We would like to thank Botong Wang for helpful comments.
This work is supported by the National Key Research and Development Program of China (No. 2021YFA1002300) and National Natural Science Foundation of China (No. 11901336).

\section{Preliminaries}\label{sec pre}

In this section, we present some basic definitions, notations and results that will be used in the paper.

We denote by
\begin{itemize}
  \item $\Lambda^{p,q}(\mathbb{C}^n)$ the space of $(p, q)$ forms on $\mathbb{C}^n$ with constant coefficients;
  \item  $\Lambda^{p,p}_{\mathbb{R}}(\mathbb{C}^n) \subset \Lambda^{p,p}(\mathbb{C}^n)$ the subspace of real $(p, p)$ forms.
\end{itemize}
Fix a coordinate system $(z_1,...,z_n)$ on $\mathbb{C}^n$, a real $(1,1)$ form
$$\alpha = \mathrm{i} \sum_{i,j=1}^n a_{ij} dz_i \wedge d\overline{z}_j$$
is called
\begin{itemize}
  \item \emph{semi-positive} if the $n\times n$ Hermitian matrix $[a_{ij}]$ is semi-positive;
  \item \emph{K\"ahler} if $[a_{ij}]$ is positive definite.
\end{itemize}
We sometimes identify the $(1,1)$ form $\alpha$ with the matrix $[a_{ij}]$, for example, $\ker \alpha$ means the kernel of the linear transform associated to $[a_{ij}]$.


Fix a K\"ahler metric $\omega$ with constant coefficients on $\mathbb{C}^n$. A real $(1,1)$ form $\alpha \in \Lambda^{1,1}_{\mathbb{R}}(\mathbb{C}^n)$ is called
\begin{itemize}
  \item \emph{$m$ positive} with respect to $\omega$, if
  $\alpha^k \wedge \omega^{n-k} >0$
for every $k$ with $1\leq k\leq m$.
\end{itemize}
This kind of positivity notion arises naturally in the study of Hessian equations.
Assume that
$$\omega=\mathrm{i}\sum_{i=1} ^n dz_i \wedge d\overline{z}_i,\ \alpha= \mathrm{i} \sum_{i=1} ^n \lambda_i dz_i \wedge d\overline{z}_i,$$
then $\alpha$ is $m$ positive with respect to $\omega$ if and only if
$$s_k (\lambda_1,...,\lambda_n) >0,\ 1\leq k\leq m,$$
where $s_k$ is the $k$-th elementary symmetric polynomial with $n$ variables.

The form $\alpha$ is called
\begin{itemize}
  \item \emph{exactly $m$ positive} if it is $m$ positive but not $m+l$ positive for any $l\geq 1$.
\end{itemize}

It is easy to see:

\begin{lem}
Assume that $\alpha \in \Lambda^{1,1}_{\mathbb{R}}(\mathbb{C}^n)$ is semi-positive, then
$\alpha$ is $m$ positive if and only if one of the following equivalent conditions holds:

\begin{enumerate}
  \item $\alpha$ has no less than $m$ positive eigenvalues;
  \item $\alpha ^m \wedge \omega^{n-m} > 0$;
  \item $\codim \ker \alpha \geq m$, where
  \begin{align*}
    \ker \alpha &=\{v\in \mathbb{C}^n|\alpha(v,w)=0, \forall w \in \mathbb{C}^n\}\\
    &=\{v\in \mathbb{C}^n|\alpha(v,v)=0\}.
  \end{align*}
\end{enumerate}

\end{lem}

On a complex manifold $X$, similar positivity notions (e.g., $m$ positivity, exact $m$ positivity) can be defined if the forms satisfy the positivity requirement at every point of $X$.

Let $X$ be a compact K\"ahler manifold with a K\"ahler metric $\widehat{\omega}$ and $\alpha \in H^{1,1} (X, \mathbb{R})$, then $\alpha$ is called nef if for any $\varepsilon >0$, there exists a smooth representative $\widehat{\omega}_\varepsilon \in \alpha$ such that
$$\widehat{\omega}_\varepsilon > - \varepsilon \widehat{\omega}.$$
Equivalently, $\alpha$ is nef if and only if it is in the closure of the K\"ahler cone of $X$.  Furthermore, if $X$ is projective and $\alpha$ is a real divisor class, then $\alpha$ is nef if and only if $$\alpha \cdot [C] \geq 0$$ for any effective $1$-cycle $C$. The analytic nefness and algebraic nefness are equivalent for divisor classes on a complex projective manifold.

Let $X=\mathbb{C}^n / \Gamma$ be a compact complex torus. It is well known that, by harmonic forms with respect to the standard Euclidean metric $\omega= \mathrm{i}\sum dz_i \wedge d\overline{z}_i$ on $X$, we have
\begin{equation*}
  H^{p,q} (X, \mathbb{C}) \cong \Lambda^{p, q} (\mathbb{C}^n)
\end{equation*}
and the cup product of cohomology classes is identified with the wedge product of differential forms. A class $\alpha \in H^{1,1} (X, \mathbb{R})$ is nef if and only if it has a representative $\widehat{\alpha} \in \alpha$ such that $\widehat{\alpha}$ is a semi-positive form with constant coefficients.

Let $X$ be a compact K\"ahler manifold with a K\"ahler class $\omega$ and $\alpha \in H^{1,1} (X, \mathbb{R})$ a nef class, then the numerical dimension of $\alpha$ is given by
\begin{align*}
  \nd(\alpha)&=\max \{k| \alpha^k \cdot \omega^{n-k}>0\}\\
  &=\max \{k| \alpha^k \neq 0\}.
\end{align*}

The following lemma is clear.

\begin{lem}
Let $X=\mathbb{C}^n / \Gamma$ be a compact complex torus, and $\alpha$ a nef class on $X$. Let $\widehat{\alpha} \in \alpha$ be the semi-positive representative with constant coefficients. Then $\nd(\alpha)$ is the number of positive eigenvalues of $\widehat{\alpha}$ and is also equal to $\codim \ker \widehat{\alpha}$.

\end{lem}

A \emph{polymatroid} on a finite set $E$ is given by a rank function $r: 2^E \rightarrow \mathbb{Z}_{\geq 0}$ satisfying the following axioms:
\begin{itemize}
  \item (Submodularity) For any $A_1, A_2 \subset E$, we have $r(A_1 \cup A_2) + r(A_1 \cap A_2) \leq r(A_1) + r(A_2)$;
  \item (Monotonicity) For any $A_1 \subset A_2 \subset E$, we have $r(A_1) \leq r(A_2)$;
  \item (Normalization) For the empty set $\emptyset$, $r(\emptyset) =0$.
\end{itemize}

A polymatroid is called loopless, if the rank of any nonempty subset is nonzero.

A polymatroid is a \emph{matroid} if it also satisfies the ``boundedness'' axiom: for any $A\subset E$, we have $$r(A)\leq |A|.$$

A discrete polymatroid $\mathcal{P}$ on $[m]$ with the rank function $r$ is a collection of points in $\mathbb{N}^m$ of the following form
\begin{equation*}
  \mathcal{P} = \{\mathbf{n}=(n_1,...,n_m)\in \mathbb{N}^m | n_{[m]} = r([m])\ \text{and}\ n_I \leq r(I),\ \forall I\subsetneq [m] \}.
\end{equation*}

\begin{exmple}(see \cite[Definition 2.16]{CCLMZ20})\label{matroid exmpl}
Let $V_1,...,V_m$ be linear subspaces of a vector space $V$ over some field $k$, then the rank function
\begin{equation*}
  r(I) =\dim (\sum_{i\in I} V_i),\ I\subset [m],
\end{equation*}
defines a polymatroid on $[m]$.
A matroid/polymatroid of this kind is called linear or representable over the field $k$.

Let $k\hookrightarrow L$ be a field extension and let $k\hookrightarrow L_i \subset L$, $i\in [m]$, be intermediate field extensions, then the rank function
\begin{equation*}
  r(I)=\trdeg_k (\wedge_{i\in I} L_i)
\end{equation*}
defines a polymatroid, where $\wedge_{i\in I} L_i$ is the smallest subfield in $L$ containing all $L_i$.
A matroid/polymatroid of this type is called algebraic over the field $k$.

Over a field of characteristic zero, the algebraic and representable polymatroids coincide \cite{ingletonRepreMatro}.

The Vam\'{o}s matroid $V_8$ is not representable over any field (see e.g. \cite{oxleyMatroidBook}), indeed, it is not algebraic over any field (see \cite{ingletonNonalgebraic}).
\end{exmple}

In the study of combinatorial properties of numerical dimensions, the ``reverse Khovanskii-Teissier inequalities'' noted in \cite{lehXiaoCorrespondences} (see also \cite{XiaoWeakMorse},\cite{dangxiao-valuations}) will be used in Section \ref{sec polymatroid}.

\begin{lem}\label{reverse KT}
Let $X$ be a compact K\"ahler manifold of dimension $n$, and let $A_1,...,A_k, B, C_1,...,C_{n-k}$ be nef $(1,1)$ classes. Then the following inequality holds:
\begin{equation*}
  \frac{n!}{k!(n-k)!} (A_1 \cdot...\cdot A_k\cdot B^{n-k}) (B^k \cdot C_1\cdot...\cdot C_{n-k}) \geq (B^n) (A_1 \cdot...\cdot A_k\cdot C_1\cdot...\cdot C_{n-k}).
\end{equation*}

\end{lem}

The terminology ``reverse Khovanskii-Teissier inequalities'' was first introduced in the work \cite{lehmxiao16convexity}, due to the fact that the classical Khovanskii-Teissier inequalities give a lower bound of the intersection number $A_1 \cdot...\cdot A_k\cdot C_1\cdot...\cdot C_{n-k}$ and our result gives an upper bound by interpolating an auxiliary nef class $B$.

\begin{rmk}\label{jiangli reverseKT}
For a projective variety over an arbitrary algebraically closed field, analogous estimates were established in \cite{jiang2021algebraicKT} by using multipoint Okounkov bodies.
\end{rmk}

\section{Hard Lefschetz property}\label{sec hardlef}

In this section, we give a sufficient and necessary characterization on complete intersection classes of nef classes to have HL property on a compact complex torus.

Recall that we are going to prove:

\begin{thrm} [=Theorem \ref{hardlef tori}] 
Let $X=\mathbb{C}^n / \Gamma$ be a compact complex torus of dimension $n$ and $0\leq p,q\leq p+q\leq n$. Let $\alpha_1,...,\alpha_{n-p-q}, \eta\in H^{1,1} (X, \mathbb{R})$ be nef classes on $X$. Denote $\Omega=\alpha_1\cdot...\cdot\alpha_{n-p-q}$. Then for HL property, the following statements are equivalent:
\begin{enumerate}
  \item the intersection class $\Omega$ has HL property;
  \item for any subset $I \subset [n-p-q]$, $\nd(\alpha_I)\geq |I|+p+q$.
\end{enumerate}
For HR property, the following statements are equivalent:
\begin{enumerate}
  \item the pair $(\Omega, \eta)$ has HR property for $\eta$ with $\nd(\eta)\geq p+q$;
  \item for any subset $I \subset [n-p-q]$, $\nd(\alpha_I)\geq |I|+p+q$.
\end{enumerate}

\end{thrm}

The result is inspired by \cite{Panov1987ONSP} and \cite{xiaoHodgeIndex, xiaomixedHRR}.

As explained in Section \ref{sec pre}, on a compact complex torus this is equivalent to the K\"ahler package in the linear setting: we identify $H^{p,q} (X, \mathbb{C})$ with $\Lambda^{p,q} (\mathbb{C}^n)$ and identify nef classes with semi-positive forms with constant coefficients on $\mathbb{C}^n$. In the following, we use these identifications.

We first show that the positivity assumptions imply HL or HR property.

\begin{prop}\label{thrm HRR Linear}

Let $0 \leq p,q \leq p+q \leq n$ and $\alpha_1,...,\alpha_{n-p-q},\eta \in \Lambda^{1,1}_{\mathbb{R}}(\mathbb{C}^n)$ be semi-positive $(1,1)$ forms.
Denote $\Omega=\alpha_1\wedge...\wedge\alpha_{n-p-q}$.
Assume that $\eta$ is $p+q$ positive and $\alpha_{I}$ is $|I|+p+q$ positive for any subset $I \subset [n-p-q]$. Then the followings holds:
\begin{enumerate}
\item The form $\Omega$ has HL property, i.e., the linear map given by wedge product with $\Omega$:
  \begin{equation*}
    \Omega : \Lambda^{p,q}(\mathbb{C}^n) \rightarrow \Lambda^{n-q,n-p}(\mathbb{C}^n)
  \end{equation*}
  is an isomorphism.

  \item The pair $(\Omega, \eta)$ has HR property, i.e., the quadratic form
  \begin{equation*}
    Q(\Phi,\Psi)=c_{p,q}\Omega \wedge \Phi \wedge \overline{\Psi}
  \end{equation*}
  is positive definite on the primitive space
  $$P^{p,q}(\mathbb{C}^n)=\{\Phi \in \Lambda^{p,q}(\mathbb{C}^n)| \Omega \wedge \eta \wedge \Phi =0\}.$$

\end{enumerate}

\end{prop}

It is clear that the HR property implies the HL property.

Denote the HL and HR statements of Proposition \ref{thrm HRR Linear} by $\HL_n$ and $\HR_n$ respectively.
We prove the result by induction on dimension, which is inspired by \cite{timorinMixedHRR} (see also \cite{xiaomixedHRR}).

We first show that $\HR_{n-1} \Rightarrow \HL_n$, whose proof is carried out in Lemma \ref{HRR_{n-1}ToHL_nStep1} and Lemma \ref{HRR_{n-1}ToHL_nStep2}.

\begin{lem}\label{HRR_{n-1}ToHL_nStep1}

Let $\alpha_1,...,\alpha_{n-p-q}$ be real $(1,1)$ forms as in Proposition \ref{thrm HRR Linear}.
Assume $\HR_{n-1}$, then for any $\Phi \in \Lambda^{p,q}(\mathbb{C}^n)$ satisfying $$\alpha_1\wedge ... \wedge \alpha_{n-p-q}\wedge \Phi=0$$ and
for any hyperplane $H \subset \mathbb{C}^n$, we have that
  \begin{equation*}
    Q_{H}(\Phi_{|H},\Phi_{|H}):=c_{p,q}{\alpha_{1}}_{|H}\wedge...\wedge{\alpha_{n-p-q-1}}_{|H}\wedge \Phi_{|H} \wedge \overline{\Phi_{|H}} \geq 0.
  \end{equation*}

\end{lem}

\begin{proof}
Let
  \begin{equation*}
    \begin{aligned}
      \mathcal{I}=& \{ I \subset [n-p-q-1]|\dim \ker \alpha_I =n-p-q-|I| \} \\
      =&\{ I \subset [n-p-q-1]|\alpha_I \ \text{is exactly} \ p+q+|I| \ \text{positive}\}.
    \end{aligned}
  \end{equation*}

We first consider the case when $\mathcal{I}\neq \emptyset$.

In this case, note that for any $I \in \mathcal{I}$, we have $$|I|\leq n-p-q-1.$$ It follows that
\begin{equation*}
  \dim \ker \alpha_I =n-p-q-|I| \geq 1.
\end{equation*}
We take $0\neq v_{I} \in \ker \alpha_I$ for each $I \in \mathcal{I}$.

Let
  \begin{equation*}
    \mathcal{H}=\{H \in \Gr^1(\mathbb{C}^n)| v_I \notin H, \forall I \in \mathcal{I} \}.
  \end{equation*}

We claim that for any $H \in \mathcal{H}$, the forms
  \begin{equation*}
    ({\alpha_1}_{|H},...,{\alpha_{n-p-q-1}}_{|H},{\alpha_{n-p-q}}_{|H})
  \end{equation*}
satisfies the positivity conditions of $\HR_{n-1}$, i.e.,
\begin{itemize}
  \item  ${\alpha_{n-p-q}}_{|H}$ is $p+q$ positive;
  \item ${\alpha_I}_{|H}$ is $|I|+p+q$ positive for any $ I \subset [n-p-q-1]$.
\end{itemize}
The former is obvious, since
  $\alpha_{n-p-q}$ is $p+q+1$ positive by our assumption and $H$ is a hyperplane. To show the latter, we consider two cases.
If $I \in \mathcal{I}$, then
  \begin{equation*}
    \ker {\alpha_I}_{|H}=H\cap \ker \alpha_I \subsetneq \ker \alpha_I,
  \end{equation*}
since $H$ does not contain a non-zero vector $v_I$ of $\ker \alpha_I$.
Hence
  \begin{align*}
    \dim \ker {\alpha_I}_{|H} &=\dim \ker \alpha_I-1 \\
    &=(n-1)-p-q-|I|,
  \end{align*}
which implies that ${\alpha_I}_{|H}$ is $p+q+|I|$ positive.
If $I \notin \mathcal{I}$, then $\alpha_I$ is $p+q+|I|+1$ positive by the definition of $\mathcal{I}$. Then ${\alpha_I}_{|H}$ is
automatically $p+q+|I|$ positive.

This finishes the proof of the claim.

Still suppose $H \in \mathcal{H}$. By assumption, $${\alpha_1}_{|_H}\wedge ... \wedge {\alpha_{n-p-q}}_{|_H}\wedge \Phi_{|_H}=0.$$
Using the above claim and the assumption $\HR_{n-1}$, we see that
$$Q_H(\Phi_{|H},\Phi_{|H})\geq 0.$$

On the other hand, $\mathcal{H}$ is nothing but a projective space removing finitely many proper linear subspaces. Then $\overline{\mathcal{H}}=\Gr^1(\mathbb{C}^n)$.
By continuity, we get that
$$Q_H(\Phi_{|H},\Phi_{|H})\geq 0,$$
for any hyperplane $H \subset \mathbb{C}^n$.

Now we consider the remaining case when $\mathcal{I}=\emptyset$. By the definition of $\mathcal{I}$, in this case, for any $I\subset [n-p-q]$, $\alpha_I$ is at least $p+q+|I|+1$ positive. For any $H \in \Gr^1(\mathbb{C}^n)$, the forms
  \begin{equation*}
    ({\alpha_1}_{|H},...,{\alpha_{n-p-q-1}}_{|H},{\alpha_{n-p-q}}_{|H})
  \end{equation*}
satisfy the positivity condition of $\HR_{n-1}$,
thus the above discussion still holds.

This finishes the proof.

\end{proof}

\begin{lem}\label{HRR_{n-1}ToHL_nStep2}
Let $\alpha_1,...,\alpha_{n-p-q}$ be real $(1,1)$ forms as in Proposition \ref{thrm HRR Linear}.
Assume $\HR_{n-1}$, then any $\Phi \in \Lambda^{p,q}(\mathbb{C}^n)$ satisfying $$\alpha_1\wedge ... \wedge \alpha_{n-p-q}\wedge \Phi=0$$
vanishes.
\end{lem}

\begin{proof}

Let
  \begin{equation*}
    \mathcal{J}= \{ I \subset [n-p-q-1]|\ker \alpha_I \not\subset \ker \alpha_{n-p-q}\}.
  \end{equation*}

As in Lemma \ref{HRR_{n-1}ToHL_nStep1}, we first consider the case when $\mathcal{J}\neq \emptyset$.

In this case, take a
  \begin{equation*}
    v_I \in \ker \alpha_I \backslash \ker \alpha_{n-p-q}
  \end{equation*}
for each $I \in \mathcal{J}$.

Let
  \begin{equation*}
    \mathcal{S}=\{H \in \Gr^1(\mathbb{C}^n)| \ker \alpha_{n-p-q} \subset H \ \text{and} \ v_I \notin H , \forall I \in \mathcal{J} \}.
  \end{equation*}

We first show that $\Phi_{|_H}=0$ for any $H\in \mathcal{S}$.

In fact, for any $H \in \mathcal{S}$, since $\ker \alpha_{n-p-q} \subset H$ we can diagonalize $\alpha_{n-p-q}$ in the following way:
  \begin{equation*}
    \alpha_{n-p-q}=\mathrm{i} \gamma_1  dH \wedge d\overline{H}+\mathrm{i}\sum_{j=2}^{n}\gamma_j  dH_j \wedge d\overline{H_j},
  \end{equation*}
with $\gamma_1 >0$ and $\gamma_j \geq 0 $ for $2\leq j \leq n$. Here, $dH$ denotes the direction orthogonal to the plane $H$.

Then $\Omega \wedge \Phi =0$ implies that
    \begin{align*}
      0 &= c_{p,q}\Omega \wedge \Phi \wedge \overline{\Phi} \\
      &= c_{p,q} \alpha_1 \wedge ...\wedge \alpha_{n-p-q-1} \wedge \Phi \wedge \overline{\Phi} \wedge
      (\mathrm{i} \gamma_1  dH \wedge \overline{dH}+\mathrm{i}\sum_{j=2}^{n}\gamma_j  dH_j \wedge \overline{dH_j}) \\
      &= \gamma_1 Q_{H}(\Phi_{|H},\Phi_{|H})+\sum_{j=2}^{n}\gamma_j Q_{H_j}(\Phi_{|{H_j}},\Phi_{|{H_j}})
    \end{align*}

Lemma \ref{HRR_{n-1}ToHL_nStep1} implies that
  $$Q_{H}(\Phi_{|H},\Phi_{|H})\geq 0$$
  and
  $$ Q_{H_j}(\Phi_{|{H_j}},\Phi_{|{H_j}}) \geq 0$$
for any $2 \leq j \leq n$.
Then all the summands of the above equation must be zero. Together with $\gamma_1 >0$, we get
  \begin{equation*}
   Q_{H}(\Phi_{|H},\Phi_{|H})=0
  \end{equation*}

Moreover, it is not hard to check that the forms ${\alpha_1}_{|H},...,{\alpha_{n-p-q-1}}_{|H},{\alpha_{n-p-q}}_{|H}$ have the desired positivity for $\HR_{n-1}$:
\begin{itemize}
  \item it is clear that ${\alpha_{n-p-q}}_{|H}$ is $p+q$ positive on $H$;
  \item for $I \subset [n-p-q-1]$, if $I \in \mathcal{J}$, since $v_I \in \ker \alpha_I$ and $v_I \notin H$, ${\alpha_{I}}_{|H}$ is $|I|+p+q$ positive;
  \item for $I \subset [n-p-q-1]$, if $I \notin \mathcal{J}$, we have $\ker \alpha_I \subset \ker \alpha_{n-p-q}$, which implies that $$\ker \alpha_{I \cup \{n-p-q\}} = \ker \alpha_I.$$
      Since $\alpha_{I \cup \{n-p-q\}}$ is $|I|+p+q+1$ positive by assumption, $\alpha_I$ is also $|I|+p+q+1$ positive, yielding that ${\alpha_{I}}_{|H}$ is $|I|+p+q$ positive.
\end{itemize}

It follows from $\HR_{n-1}$ that $\Phi_{|H}=0$ for any $H\in \mathcal{S}$, as desired.

By the definition of $\mathcal{S}$, we have
  \begin{equation*}
    \mathcal{S}=\{H \in \Gr^1(\mathbb{C}^n)| \ker \alpha_{n-p-q} \subset H  \}\backslash
    \bigcup_{I\in \mathcal{J}} \{  H \in \Gr^1(\mathbb{C}^n)| v_I \in H  \} .
  \end{equation*}
Note that $v_I \notin \ker \alpha_{n-p-q}$ for any $I\in \mathcal{J}$. Hence $\{  H \in \Gr^1(\mathbb{C}^n)| v_I \in H  \}$ intersects
  properly with $$\{H \in \Gr^1(\mathbb{C}^n)| \ker \alpha_{n-p-q} \subset H  \}.$$
Therefore,
  \begin{equation*}
    \overline{\mathcal{S}}=\{H \in \Gr^1(\mathbb{C}^n)| \ker \alpha_{n-p-q} \subset H  \}.
  \end{equation*}

By continuity, we get that $\Phi_{|H}=0$ for any $H  \supset \ker \alpha_{n-p-q}$.

By assumption, $\alpha_{n-p-q}$ is $p+q+1$ positive, that is,
$$\dim \ker \alpha_{n-p-q} \leq n-1-p-q. $$
Hence for any $w_1,...,w_{p+q} \in \mathbb{C}^n$,
  \begin{equation*}
    \dim \Span\{w_1,...,w_{p+q},\ker \alpha_{n-p-q}\} \leq n-1.
  \end{equation*}
  We can extend this linear subspace to a hyperplane $H \in \overline{\mathcal{S}}$. Then
  \begin{equation*}
    \Phi(w_1,...,w_{p+q})=\Phi_{|H} (w_1,...,w_{p+q})=0.
  \end{equation*}
  Since $w_1,...,w_{p+q} \in \mathbb{C}^n$ are arbitrary, we get $\Phi=0$.

In the case when $\mathcal{J}= \emptyset$, we let
 \begin{equation*}
    \mathcal{S}=\{H \in \Gr^1(\mathbb{C}^n)| \ker \alpha_{n-p-q} \subset H \}.
  \end{equation*}
It is not hard to check that for any $H\in \mathcal{S}$, the forms ${\alpha_1}_{|H},...,{\alpha_{n-p-q-1}}_{|H},{\alpha_{n-p-q}}_{|H}$ have the desired positivity for $\HR_{n-1}$:
\begin{itemize}
  \item ${\alpha_{n-p-q}}_{|H}$ is $p+q$ positive on $H$;
  \item for any $I \subset [n-p-q-1]$, we have $\ker \alpha_I \subset \ker \alpha_{n-p-q}$, which implies that ${\alpha_{I}}_{|H}$ is $|I|+p+q$ positive.
\end{itemize}
Thus the above argument for $\mathcal{J}\neq \emptyset$ works in the same way.

This finishes the proof.
\end{proof}

Next we show that $\HL_n \Rightarrow \HR_n$, whose proof is carried out in Lemma \ref{HLtoLD} and Lemma \ref{HLtoHR}.

\begin{lem}\label{HLtoLD}
Notations and assumptions as in Proposition \ref{thrm HRR Linear},
assume $\HL_n$, then we have that
\begin{itemize}
  \item the Lefschetz decomposition holds, i.e., the
$Q$-orthogonal decomposition:
  \begin{equation*}
    \Lambda^{p,q}(\mathbb{C}^n)=\eta \wedge \Lambda^{p-1,q-1}(\mathbb{C}^n)\oplus P^{p,q}(\mathbb{C}^n),
  \end{equation*}
where we use the convention that $\Lambda^{p-1,q-1}(\mathbb{C}^n) =\{0\}$ when $p=0$ or $q=0$.
  \item $\dim P^{p,q}(\mathbb{C}^n) =\dim \Lambda^{p,q}(\mathbb{C}^n) - \dim \Lambda^{p-1,q-1}(\mathbb{C}^n)$.
\end{itemize}

\end{lem}

\begin{proof}
It is clear that $\alpha_1,...,\alpha_{n-p-q},\eta,\eta$ have the desired positivity for $\HL_n$ on $\Lambda^{p-1,q-1}(\mathbb{C}^n)$, therefore by the assumption $\HL_n$,
  \begin{equation*}
    \Omega \wedge \eta^2: \Lambda^{p-1,q-1}(\mathbb{C}^n) \rightarrow \Lambda^{n-q+1,n-p+1}(\mathbb{C}^n)
  \end{equation*}
  is a linear isomorphism.
  In particular,
  \begin{equation*}
    \eta:\Lambda^{p-1,q-1}(\mathbb{C}^n) \rightarrow \Lambda^{p,q}(\mathbb{C}^n)
  \end{equation*}
  is an injection and
  \begin{equation*}
    \Omega \wedge \eta:\Lambda^{p,q}(\mathbb{C}^n) \rightarrow \Lambda^{n-q+1,n-p+1}(\mathbb{C}^n)
  \end{equation*}
  is a surjection with $\ker\Omega \wedge \eta=P^{p,q} (\mathbb{C}^n)$.

Again, since $\Omega \wedge \eta^2$ is an isomorphism, we have
  \begin{equation*}
    \eta\wedge \Lambda^{p-1,q-1}(\mathbb{C}^n) \cap P^{p,q} (\mathbb{C}^n) =\{0\}.
  \end{equation*}

Take the dimensions into account, one finds that
  \begin{equation*}
    \Lambda^{p,q}(\mathbb{C}^n)=\eta \wedge \Lambda^{p-1,q-1}(\mathbb{C}^n)\oplus P^{p,q}(\mathbb{C}^n).
  \end{equation*}

The orthogonality is straightforward by definition.

Finally the identity for $\dim P^{p,q} (\mathbb{C}^n)$ follows from the discussions above.
\end{proof}

\begin{lem}\label{HLtoHR}
Notations and assumptions as in Proposition \ref{thrm HRR Linear},
assume $\HL_n$, then the quadratic form
  \begin{equation*}
    Q(\Phi,\Psi)=c_{p,q}\Omega \wedge \Phi \wedge \overline{\Psi}
  \end{equation*}
is positive definite on the primitive space
  $$P^{p,q}(\mathbb{C}^n)=\{\Phi \in \Lambda^{p,q}(\mathbb{C}^n)| \Omega \wedge \eta \wedge \Phi =0\}.$$
\end{lem}

\begin{proof}
Fix a K\"ahler form $\omega$ and for $t\in [0,1]$ set
  \begin{equation*}
    \alpha_j^t=(1-t)\alpha_j+t\omega, \eta^t=(1-t)\eta +t\omega,\Omega^t=\alpha_1^t \wedge...\wedge \alpha_{n-p-q}^t
  \end{equation*}
  and
  \begin{equation*}
    Q_t(\Phi,\Psi)= c_{p,q}\Omega^t \wedge \Phi\wedge \overline{\Psi}.
  \end{equation*}

By Lemma \ref{HLtoLD}, we have the $Q_t$-orthogonal decomposition:
  \begin{equation*}
    \Lambda^{p,q}(\mathbb{C}^n)=\eta^t \wedge \Lambda^{p-1,q-1}(\mathbb{C}^n) \oplus P^{p,q}_t,
  \end{equation*}
  where
  \begin{equation*}
    P^{p,q}_t=\{ \Phi \in \Lambda^{p,q}(\mathbb{C}^n) | \Omega^t \wedge \eta^t \wedge \Phi =0\}.
  \end{equation*}

By Lemma \ref{HLtoLD} again, the dimension of $P^{p,q}_t$ is stable under deformation.

Since $Q_1$ is positive definite on $P^{p,q}_1$ by the classical Hodge theory and the quadratic form $Q_t$ is non-degenerate in the course of deformation by the assumption $\HL_n$,  it follows that $Q_0$ is positive definite on $P^{p,q}_0$.

This finishes the proof.
\end{proof}

We have finished the proof that the positivity assumption implies HL or HR property.

Next we show that the the positivity assumption is also necessary.

\begin{prop}\label{necessary condition of HL}
  Let $0\leq p,q\leq p+q \leq n$ and $\alpha_1,...,\alpha_{n-p-q}\in \Lambda^{1,1}_{\mathbb{R}}(\mathbb{C}^n)$ be semi-positive $(1,1)$ forms.
  If there exists some $I \subset [n-p-q]$ such that $\alpha_I$ is not $|I|+p+q$ positive, then the form $$\alpha_1\wedge...\wedge \alpha_{n-p-q}$$ does not have HL property.
\end{prop}

\begin{proof}
We may suppose $I=[k]$ for some $1\leq k \leq n-p-q$ and assume that $$\alpha_{[k]}=\alpha_1+...+\alpha_k$$ is not $k+p+q$ positive.

Without loss of generalities, by taking a coordinates transform we can diagonalize $\alpha_{[k]}$ as following:
  \begin{equation*}
    \alpha_{[k]}=\mathrm{i}\sum_{j=1}^{k+p+q-1}\gamma_j dz_j\wedge d\overline{z_j},\ \gamma_j \geq 0.
  \end{equation*}
Since $\ker \alpha_{[k]} \subset \ker \alpha_j$ for any $j \in I$, the product of $\alpha_j$ takes the form
  \begin{equation*}
    \alpha_1 \wedge...\wedge \alpha_k = \sum_{\substack{I,J\subset [k+p+q-1] \\ |I|=|J|=k}}u_{I,J} dz_I\wedge d\overline{z_J}.
  \end{equation*}

Consider the linear subspaces:
\begin{align*}
  H_1 = \Span \{dz_I\wedge d\overline{z_J}| I,J \subset [k+p+q-1] ,|I|=p,|J|=q \}
  \subset \Lambda^{p,q}(\mathbb{C}^n)
\end{align*}
and
\begin{align*}
   H_2 =\Span \{dz_I\wedge d\overline{z_J}| I,J \subset [k+p+q-1] ,|I|=p+k,|J|=q+k \}
   \subset \Lambda^{p+k,q+k}(\mathbb{C}^n).
\end{align*}

Note that $H_1$ and $H_2$ are stablized by $\alpha_1 \wedge...\wedge \alpha_k$ in the sense that
\begin{equation*}
  \alpha_1 \wedge...\wedge \alpha_k \wedge H_1 \subset H_2.
\end{equation*}

When $p=0$ or $q=0$, it is easy to see that $\alpha_1 \wedge...\wedge \alpha_k \wedge H_1 =\{0\}$, therefore $\alpha_1 \wedge...\wedge \alpha_k$ is not injective on $\Lambda^{p,q}(\mathbb{C}^n)$.

Next we consider the case when $p\geq 1, q\geq 1$.
Note that
\begin{align*}
  \dim H_1 &= \binom{k+p+q-1}{p}\binom{k+p+q-1}{q},\\
  \dim H_2 &= \binom{k+p+q-1}{p+k}\binom{k+p+q-1}{q+k}.
\end{align*}
It follows that $$\dim H_1 > \dim H_2$$
and thus
\begin{equation*}
  \alpha_1 \wedge...\wedge \alpha_k: \Lambda^{p,q}(\mathbb{C}^n)\rightarrow \Lambda^{p+k,q+k}(\mathbb{C}^n)
\end{equation*}
is not injective. A fortiori, $\alpha_1\wedge...\wedge \alpha_{n-p-q}$ is not injective.
\end{proof}

Together with Proposition \ref{thrm HRR Linear} and Proposition \ref{necessary condition of HL}, we finish the proof of Theorem \ref{hardlef tori}.

\begin{rmk}\label{panov summary}
Let $\Herm_n$ be the space of $n\times n$ Hermitian matrices, and let $A_1,...,A_n \in \Herm_n$. The mixed discriminant of $A_1,...,A_n$ is given by
  \begin{equation*}
    D(A_1,...,A_n)=\frac{1}{n!}\frac{\partial^n }{\partial t_1... \partial t_n} \det(t_1A_1+... + t_nA_n).
  \end{equation*}

The matrices $A_1,...,A_{n-2}$ induce a quadratic form $Q$ on $\Herm_n$:
  \begin{align*}
    Q:&\Herm_n\times \Herm_n \rightarrow \mathbb{R}\\
    &(A,B) \mapsto D(A,B,A_1,...,A_{n-2})
  \end{align*}

If all $A_1,...,A_{n-2}$ are positive definite, then $Q$ is a Lorentzian form (i.e., a quadratic form with signature $(+,-,\cdots,-)$), which was shown in the seminal work of Alexandrov \cite{alexandroff1938theorie} (see also \cite{alexandrovSelectedworksIconvexbodies}).

Panov \cite{Panov1987ONSP} generalized this important result to semi-positive real symmetric matrices $A_1,...,A_{n-2}$ under certain non-degenerate assumption and proved that the assumption is also necessary. While Panov's results are only stated for real symmetric matrices, it also works for Hermitian matrices without any difficulty.

We give a brief summary of Panov's work.
Denote the number of positive eigenvalues of a semi-positive matrix $A$ by $r(A)$.
Assume that $A_1,...,A_{n}\in \Herm_n$ are semi-positive, then the following equivalence holds:
\begin{itemize}
    \item $D(A_1 ,A_2,...,A_{n})>0$;
    \item $r(A_I)\geq |I|,\ \text{for any}\ I\subset [n]$.
\end{itemize}
and the following statements are equivalent:
  \begin{itemize}
    \item $Q=D(- ,- ,A_1,...,A_{n-2})$ is a Lorentzian form on $\Herm_n$;
    \item $r(A_I)\geq |I|+2,\ \text{for any}\ I \subset [n-2]$.
  \end{itemize}
Furthermore, the work \cite{Panov1987ONSP} also give a thorough description of the degeneration of $Q$ when $A_1,...,A_{n-2}$ do not satisfy the above non-degenerate condition.

Note that we have the isomorphism
\begin{align*}
    \phi: &\Herm_n \rightarrow \Lambda^{1,1}_{\mathbb{R}}(\mathbb{C}^n),\\
    &A=[a_{jk}] \mapsto \mathrm{i} \sum_{j,k =1} ^n a_{jk}dz_j \wedge d\overline{z_k}.
\end{align*}
Denote $\alpha_k = \phi (A_k)$, then we have (up to a constant volume form on $\mathbb{C}^n$)
\begin{equation*}
  D(A_1,...,A_n)= \alpha_1 \wedge ...\wedge \alpha_n.
\end{equation*}

By the above identification, one can use the $p=q=1$ and $p=q=0$ cases of Theorem \ref{hardlef tori} to deduce Panov's results.
\end{rmk}

In the following, we discuss some consequences of Theorem \ref{hardlef tori}.

The first is an application to the case $p=q=0$ that we mentioned in Remark \ref{panov summary}. We restate it using positivity. Note that the semi-positive $(1,1)$ forms $\alpha_1,...,\alpha_n \in \Lambda^{1,1}_{\mathbb{R}}(\mathbb{C}^n)$ has HL property if and only if
$$\alpha_1\wedge ... \wedge \alpha_n >0.$$
Therefore, applying Theorem \ref{hardlef tori}, we get a description of positivity of complete intersection of semi-positive real $(1,1)$ forms.

\begin{cor}\label{interPos linear}
  Let $\alpha_1,... ,\alpha_n \in \Lambda^{1,1}_{\mathbb{R}}(\mathbb{C}^n)$ be semi-positive $(1,1)$ forms. Then the following statements are equivalent:
  \begin{itemize}
    \item $\alpha_1\wedge ... \wedge \alpha_n >0$
    \item $\alpha_I$ is $|I|$ positive, for any $I \subset [n]$.
  \end{itemize}
\end{cor}

The second is an application to the products of forms satisfying HL property.

\begin{cor}
Let $\alpha_1,...,\alpha_k, \beta_1,...,\beta_l \in \Lambda^{1,1}_{\mathbb{R}}(\mathbb{C}^n)$ be semi-positive $(1,1)$ forms with $k+l\leq n$. Denote $\Phi=\alpha_1 \wedge...\wedge\alpha_k$ and $\Psi=\beta_1 \wedge...\wedge\beta_l$. Assume that $\Phi, \Psi$ have HL property, then $\Phi \wedge \Psi$ also has HL property.

\end{cor}

\begin{proof}

By Theorem \ref{hardlef tori}, $\Phi$ has HL property on $\Lambda^{n-k}(\mathbb{C}^n)$ if and only if for any subset $I\subset [k]$, $$\dim \ker \alpha_I \leq k-|I|.$$
Similarly, $\Psi$ has HL property on $\Lambda^{n-l}(\mathbb{C}^n)$ if and only if for any subset $J\subset [l]$,
$$\dim \ker \beta_J \leq l-|J|.$$

Therefore, for any subset $I\subset [k]$ and any subset $J\subset [l]$,
\begin{equation*}
  \dim \ker (\alpha_I + \beta_J) \leq \min (k-|I|, l-|J|) \leq k+l - (|I| +|J|).
\end{equation*}
Together with
$$\dim \ker \alpha_I \leq k+l-|I|$$
for any subset $I\subset [k]$ and
$$\dim \ker \beta_J \leq k+l-|J|$$
any subset $J\subset [l]$, we obtain that $\Phi \wedge \Psi$ also has HL property on $\Lambda^{n-k-l}(\mathbb{C}^n)$ by
Theorem \ref{hardlef tori} again.

\end{proof}

Using the linear version, we obtain new kinds of cohomology classes on an arbitrary compact K\"ahler manifold, which have HL and HR properties. Recall that we need to prove:

\begin{cor}[=Corollary \ref{cor kahler hrr intro}]\label{cor global}
Let X be a compact K\"ahler manifold of dimension $n$ and $0\leq p,q\leq p+q\leq n$. Let $\alpha_1,...,\alpha_{n-p-q}, \eta \in H^{1,1}(X,\mathbb{R})$ be nef classes on $X$. Denote $\Omega=\alpha_1\cdot...\cdot\alpha_{n-p-q}$.
Assume that there exists a smooth representative $\widehat{\alpha}_i$ in each class $\alpha_i$, such that for any subset $I \subset [n-p-q]$, $\widehat{\alpha}_I$ is $|I|+p+q$ positive in the sense of forms, and that there exists a smooth representative $\widehat{\eta}$ in $\eta$, such that $\widehat{\eta}$ is $p+q$ positive, then
\begin{itemize}
  \item the complete intersection class $\Omega$ has HL property;
  \item the pair $(\Omega, \eta)$ has HR property.
\end{itemize}
\end{cor}

The proof from the local setting to the global setting is standard, following essentially from \cite{DN06} (see also \cite{xiaomixedHRR}). For reader's convenience, we give a sketch of the key argument.

We denote the space of smooth complex $(p,q)$ forms on $X$ by $\Lambda^{p,q}(X, \mathbb{C})$, and the space of primitive $(p,q)$ classes with respect to $(\Omega, \eta)$ by $P^{p,q}(X, \mathbb{C})$.

The global setting can be reduced to the local setting by solving a $dd^c$-equation.

\begin{lem}\label{ddc eq}
Notations as in Corollary \ref{cor global}, denote $\widehat{\Omega} = \widehat{\alpha}_1 \wedge...\wedge \widehat{\alpha}_{n-p-q}$, then for any smooth $d$-closed form $\widehat{\Phi} \in \Lambda^{p, q}(X, \mathbb{C})$ with its class $\{\widehat{\Phi}\}\in P^{p, q} (X, \mathbb{C})$, there is a smooth $(p-1, q-1)$ form $\widehat{F}$ such that
\begin{equation}\label{eq primitive}
  \widehat{\Omega} \wedge \widehat{\eta}\wedge dd^c \widehat{F} =  \widehat{\Omega} \wedge \widehat{\eta}\wedge \widehat{\Phi}.
\end{equation}
\end{lem}

\begin{proof}
By applying Theorem \ref{hardlef tori}, the proof is similar to \cite[Proposition 2.3]{DN06}.
\end{proof}

\begin{proof}[Proof of Corollary \ref{cor global}]

Assume that $\Phi \in P^{p, q} (X, \mathbb{C})$ and let $\widehat{\Phi}$ be a smooth representative of the class $\Phi$. Then by Lemma \ref{ddc eq}, there is a smooth form $\widehat{F}$ such that
\begin{equation}\label{eq local primitive}
  \widehat{\Omega} \wedge \widehat{\eta}\wedge (\widehat{\Phi} - dd^c \widehat{F}) =0.
\end{equation}
Thus $\widehat{\Phi} - dd^c \widehat{F}$ is a primitive $(p, q)$ form with respect to $\widehat{\Omega} \wedge \widehat{\eta}$.

Applying Theorem \ref{hardlef tori}, we have that
\begin{equation}\label{eq pt Q}
  c_{p,q}\widehat{\Omega} \wedge  (\widehat{\Phi} - dd^c \widehat{F}) \wedge \overline{(\widehat{\Phi} - dd^c \widehat{F})} \geq 0
\end{equation}
at every point.
By Stokes formula,
\begin{equation}\label{eq Q}
  Q(\Phi, \Phi) = c_{p,q}\int \widehat{\Omega} \wedge  (\widehat{\Phi} - dd^c \widehat{F}) \wedge \overline{(\widehat{\Phi} - dd^c \widehat{F})} \geq 0.
\end{equation}

Moreover, if $Q(\Phi, \Phi)=0$, then we have equalities everywhere. In particular, (\ref{eq pt Q}) is an equality at every point. By Theorem \ref{hardlef tori}, this yields
\begin{equation*}
  \widehat{\Phi} - dd^c \widehat{F} =0
\end{equation*}
on $X$. Thus $\Phi =0$ in $H^{p,q} (X, \mathbb{C})$. This finishes the proof of the global HR.

As an immediate consequence, we get the global HL. We only need to check that $$\Omega: H^{p,q} (X, \mathbb{C})\rightarrow H^{n-q,n-p} (X, \mathbb{C})$$ is injective. Assume that $\Phi \in H^{p,q} (X, \mathbb{C})$ satisfies $\Omega \cdot \Phi =0$, then $\Phi$ is primitive and $Q(\Phi, \Phi)=0$. By HR, $\Phi =0$, which finishes the proof of the HL.

This finishes the proof of Corollary \ref{cor global}.

\end{proof}

\section{Non-vanishing property}\label{sec nonvanishing}

Applying Theorem \ref{hardlef tori} to the case $p=q=0$, for nef classes $\alpha_1,...,\alpha_{n-p-q}, \eta\in H^{1,1} (X, \mathbb{R})$ on a compact complex torus the following are equivalent:
\begin{enumerate}
  \item $\alpha_1\cdot...\cdot\alpha_n >0$;
  \item for any subset $I \subset [n]$, $\nd(\alpha_I)\geq |I|$.
\end{enumerate}

It is natural to expect that this also holds on an arbitrary compact K\"ahler manifold.

In this section, we give a complete characterization on the non-vanishing of complete intersections of nef classes. We are going to prove:

\begin{thrm}[=Theorem \ref{intro thrm interPosi}]\label{thrm interPosi}
Let $X$ be a compact K\"ahler manifold of dimension $n$, $1\leq m\leq n$ and let $\alpha_1,...,\alpha_m\in H^{1,1} (X, \mathbb{R})$ be nef classes, then the following statements are equivalent:
\begin{enumerate}
  \item the intersection class $\alpha_1\cdot...\cdot\alpha_m \neq 0 $ in $H^{m,m} (X, \mathbb{R})$;
  \item for any subset $I\subset [m]$, $\nd(\alpha_I) \geq |I|$.
\end{enumerate}

\end{thrm}

The proof applies a generalized version of Hodge index theorem on compact K\"ahler manifolds, whose algebraic setting had been proved in \cite{luotieHodgeindex}. In the following, we prove a K\"ahler version by similar argument.

\begin{prop} \label{luotie}
Let $X$ be a compact K\"ahler manifold of dimension $n$, and let $\alpha_1,...,\alpha_{n-2}\in H^{1,1}(X, \mathbb{R})$ be nef classes on $X$. Denote $\Omega=\alpha_1\cdot...\cdot\alpha_{n-2}$ and $$Q(\beta_1, \beta_2)=\Omega\cdot \beta_1\cdot \beta_2$$ the quadratic form on $H^{1,1}(X, \mathbb{R})$. Assume that $\alpha$ is a real $(1,1)$ class satisfying $Q(\alpha,\alpha) > 0$, then for any $\beta \in H^{1,1}(X, \mathbb{R})$ satisfying $$Q(\alpha,\beta)=0$$
one has $Q(\beta,\beta)\leq 0$, and the equality holds if and only if $\Omega\cdot \beta =0$ in $H^{n-1,n-1}(X, \mathbb{R})$.
\end{prop}

\begin{proof}
Fix a K\"ahler class $\omega$ on $X$. For $t>0$, consider $$\Omega_t =(\alpha_1 + t\omega)\cdot...\cdot(\alpha_{n-2} +t\omega)$$
and the corresponding quadratic form $Q_t$. By \cite{DN06,cattanimixedHRR} or Corollary \ref{cor kahler hrr intro}, $Q_t$ has signature $(+,-,...,-)$, which yields that for any $\alpha', \beta' \in H^{1,1}(X, \mathbb{R})$,
\begin{equation*}
Q_t(\alpha',\beta')^2 \geq Q_t(\alpha', \alpha')Q_t(\beta',\beta').
\end{equation*}

Let $t$ tend to zero, we thus obtain the Hodge index inequality $$Q(\alpha',\beta')^2 \geq Q(\alpha', \alpha')Q(\beta',\beta').$$
Together with the assumption on $\alpha, \beta$ and letting $\alpha'=\alpha, \beta'=\beta$, we have $$Q(\beta,\beta)\leq 0.$$

This inequality implies that $Q$ is non-positive on the hyperplane
$$ P=\ker \Omega\cdot\alpha \curvearrowright H^{1,1}(X, \mathbb{R}).$$
By the assumption, $Q$ is positive definite on the subspace $\mathbb{R}\alpha$. Therefore, $Q$ has signature of the form $(+, -,...,-,0,...,0)$.

According to the signature, we fix a basis $e_k$ (with $e_1 =\alpha$), $1\leq k \leq h^{1,1}$, satisfying that $Q(e_k, e_l)=0$ for $k\neq l$, $Q(e_1, e_1)>0$, $Q(e_i, e_i)<0$ for $2 \leq i \leq s$ and $Q(e_i, -)=0$ for $s+1 \leq i \leq n$. Thus $\beta \in P$ yields that
$$\beta= \sum_{ i\geq 2} c_i e_i.$$
If we further have $Q(\beta,\beta)= 0$, then
$$\beta= \sum_{i : \ Q(e_i, -)= 0} c_i e_i,$$
implying that $\Omega \cdot \beta =0$.

This finishes the proof.
\end{proof}

As a corollary, one can get the characterization of equality of a generalized version of Teissier's proportionality problem by using $(n-1,n-1)$ classes.

\begin{cor}\label{proportional}
Let $X$ be a compact K\"ahler manifold of dimension $n$, and let $\alpha_1,...,\alpha_{n-2}\in H^{1,1}(X, \mathbb{R})$ be nef classes on $X$. Denote $\Omega=\alpha_1\cdot...\cdot\alpha_{n-2}$. Assume that $\alpha, \gamma\in H^{1,1}(X, \mathbb{R})$ are nef class satisfying that
$$Q(\alpha, \gamma)^2 = Q(\alpha, \alpha)Q(\gamma, \gamma),$$
then $\Omega \cdot \gamma$ and $\Omega \cdot \alpha$ are proportional.
\end{cor}

\begin{proof}
Without loss of generalities, we can assume that both $\Omega \cdot \gamma$ and $\Omega \cdot \alpha$ are non-zero.

If $Q(\alpha, \alpha)>0$, we just need to note that
$$\beta:= \gamma- \frac{Q(\alpha, \gamma)}{Q(\alpha, \alpha)} \alpha \in P=\ker \Omega\cdot\alpha $$
and $Q(\beta, \beta)=0$, then Proposition \ref{luotie} implies the result.

If $Q(\alpha, \alpha)=0$, then by the assumption we also have $Q(\alpha, \gamma)=0$.
If $Q(\gamma, \gamma)>0$, then we may replace $\alpha$ by $\gamma$ and the previous discussion applies.
Thus we may suppose $Q(\gamma, \gamma)=0$. Consider the subspace $\Span \{\alpha,\gamma\}$. If it is of dimension one, then there is nothing to prove. Therefore we may assume that it is of dimension two. Fix a K\"ahler class $\omega$, then $\Span \{\alpha,\gamma\}$ has a nontrivial intersection with the hyperplane $\ker \Omega\cdot\omega$, i.e., there exist two numbers $c, d$, not all zero, such that $Q(\omega, c \alpha + d\gamma)=0$. The aforementioned vanishing properties of $\alpha, \gamma$ imply that
\begin{equation*}
  Q(c \alpha + d\gamma, c \alpha + d\gamma)=0.
\end{equation*}
Then applying Proposition \ref{luotie} again shows that $\Omega\cdot (c \alpha + d\gamma)=0$, implying the result.

This finishes the proof.
\end{proof}

\begin{rmk}
We remark that Proposition \ref{luotie} and Corollary \ref{proportional} were also recently proved in \cite{zhangteiss} with a different method, and similar results for $b$-divisors in the algebraic setting were studied in \cite{dangfarveBdivisor}.
\end{rmk}

Now we can prove Theorem \ref{thrm interPosi}.

\begin{proof}[Proof of Theorem \ref{thrm interPosi}]

It is clear if the intersection class $\alpha_1\cdot...\cdot\alpha_m \neq 0 $, then $\nd(\alpha_I) \geq |I|$ for any subset $I\subset [m]$.

For the converse direction, we prove the result by induction on $m$.

Assume that the result holds for $m-1$ nef classes, in particular, we have
$$\alpha_1 \cdot ...\cdot \alpha_{m-1} \cdot \omega^{n-m+1}>0.$$

We first show that under the assumption on $I=[m]$, i.e., $$\alpha_{[m]}^m \cdot \omega^{n-m}>0,$$
we must have
\begin{equation}\label{eq m}
\alpha_1 \cdot ...\cdot \alpha_{m-1} \cdot \alpha_{[m]} \cdot \omega^{n-m} >0.
\end{equation}
Otherwise, assume that
\begin{equation}\label{eq otherwise}
  \alpha_1 \cdot ...\cdot \alpha_{m-1} \cdot \alpha_{[m]} \cdot \omega^{n-m} =0,
\end{equation}
then by Corollary \ref{proportional},
\begin{equation*}
\alpha_{[m]}\cdot \alpha_2 \cdot ...\cdot \alpha_{m-1}  \cdot \omega^{n-m} = c \alpha_1\cdot \alpha_2 \cdot ...\cdot \alpha_{m-1}  \cdot \omega^{n-m}
\end{equation*}
for some $c>0$. By (\ref{eq otherwise}), this implies that $$\alpha_{[m]}^2\cdot \alpha_2 \cdot ...\cdot \alpha_{m-1}  \cdot \omega^{n-m}=0. $$
Applying Corollary \ref{proportional} again, we get
\begin{equation*}
 \alpha_{[m]}^2\cdot \alpha_3 \cdot ...\cdot \alpha_{m-1}  \cdot \omega^{n-m} = c' \alpha_{[m]}\cdot \alpha_2 \cdot ...\cdot \alpha_{m-1}  \cdot \omega^{n-m}
\end{equation*}
for some $c'>0$, yielding that $$\alpha_{[m]}^3\cdot \alpha_3 \cdot ...\cdot \alpha_{m-1}  \cdot \omega^{n-m}=0.$$

By iteration, we obtain $$\alpha_{[m]}^m\cdot  \omega^{n-m}=0,$$
which is a contradiction with the assumption on $I=[m]$.
Therefore, we must have
\begin{equation*}
  \alpha_1 \cdot ...\cdot \alpha_{m-1} \cdot \alpha_{[m]} \cdot \omega^{n-m} >0.
\end{equation*}

We show that this implies $$\alpha_1 \cdot ...\cdot  \alpha_{m} \cdot \omega^{n-m} >0,$$
finishing the induction process.
Otherwise,
\begin{equation}\label{eq otherwise2}
  \alpha_1 \cdot ...\cdot \alpha_{m} \cdot \omega^{n-m} =0.
\end{equation}
Then by Corollary \ref{proportional}, we get that the classes $$\alpha_1 \cdot\alpha_3\cdot ...\cdot \alpha_{m} \cdot \omega^{n-m},\ \alpha_2 \cdot\alpha_3 \cdot ...\cdot \alpha_{m} \cdot \omega^{n-m}$$
are proportional. By induction on $m-1$ nef classes,
$$\alpha_1 \cdot\alpha_3\cdot ...\cdot \alpha_{m} \cdot \omega^{n-m+1}>0$$ and
$$\alpha_2 \cdot\alpha_3\cdot ...\cdot \alpha_{m} \cdot \omega^{n-m+1}>0,$$
therefore, there is some $c''>0$ such that
\begin{equation*}
\alpha_1 \cdot\alpha_3 \cdot...\cdot \alpha_{m} \cdot \omega^{n-m} = c'' \alpha_2 \cdot\alpha_3\cdot ...\cdot \alpha_{m} \cdot \omega^{n-m},
\end{equation*}
which by (\ref{eq otherwise2}) yields that $$\alpha_1 ^2 \cdot\alpha_3\cdot ...\cdot \alpha_{m} \cdot \omega^{n-m} =0.$$

The same argument shows that for every $1\leq k \leq m-1$,
\begin{equation*}
 \alpha_1 \cdot ...\cdot \alpha_{m-1} \cdot \alpha_{k} \cdot \omega^{n-m} =0.
\end{equation*}
Together with (\ref{eq otherwise2}), we get that
$$\alpha_1 \cdot ...\cdot \alpha_{m-1} \cdot \alpha_{[m]} \cdot \omega^{n-m} =0,$$
which contradicts with (\ref{eq m}).

This finishes the proof.
\end{proof}

\begin{rmk}\label{RelationWithKhovanskii'sResult}
Let $X$ be a smooth projective variety over $\mathbb{C}$.
Let $H$ be a fixed ample divisor on $X$ and $1\leq m \leq n$. Note that the numerical dimension and Iitaka dimension coincides for any semiample divisor $D$.
Then by Theorem \ref{thrm interPosi}, for any semiample divisors $D_1,...,D_m$ on $X$, the following statements are equivalent:
\begin{itemize}
  \item $D_1\cdot ...\cdot D_m \cdot H^{n-m} >0$;
  \item $\kappa(D_I)\geq |I|, \forall I\subset [m]$.
\end{itemize}
where $\kappa(D_I)$ is the Iitaka dimension of $D_I$. This recovers related results in \cite{KK16}. Indeed, by \cite{moriwakiNefGood}, a nef divisor $D$ satisfies that $\nd(D) = \kappa(D)$ if and only if $D$ is almost base point free. Therefore, the same result also extends to almost base point free divisors.

\end{rmk}

\begin{rmk}\label{GenOfThrmInterPosi}

The analogy of Theorem \ref{thrm interPosi} holds on any smooth projective variety $X$ over an algebraically closed field $k$.

More precisely, fix an ample divisor class $H$, for any nef divisor classes $\alpha_1,...,\alpha_m \in N^1(X)$, the following statements are equivalent:
\begin{itemize}
  \item the intersection product $\alpha_1\cdot...\cdot \alpha_m \cdot H^{n-m}>0$;
  \item $\nd(\alpha_I)\geq |I|$, for any $I \subset [m]$.
\end{itemize}

With the Hodge index theorem over over an algebraically closed field $k$ (see e.g. \cite[Chapter 2]{badescuSurfacebook}), the proof is exactly the same as Theorem \ref{thrm interPosi}.
\end{rmk}

By taking a resolution of singularities, Theorem \ref{intro thrm interPosi} also extends to compact K\"ahler varieties.

\begin{cor}\label{kahler variety interspos}
Let $X$ be a compact K\"ahler variety of dimension $n$, $1\leq m\leq n$ and let $\alpha_1,...,\alpha_m$ be nef classes on $X$. Fix a K\"ahler class $\omega$ on $X$, then the following statements are equivalent:
\begin{enumerate}
  \item the intersection number $\alpha_1\cdot...\cdot\alpha_m \cdot \omega^{n-m}> 0 $;
  \item for any subset $I\subset [m]$, $\nd(\alpha_I) \geq |I|$.
\end{enumerate}
\end{cor}

\begin{proof}
We only need to show that the assumption on numerical dimensions implies the non-vanishing of intersection number.

Let $\mu: \widehat{X}\rightarrow X$ be a modification such that $\widehat{X}$ is a compact K\"ahler manifold. It is clear that $\nd(\alpha_I) \geq |I|$ implies $\nd(\mu^*\alpha_I) \geq |I|$. By applying Theorem \ref{intro thrm interPosi} to $\mu^* \alpha_1,...,\mu^* \alpha_m, \mu^* \omega,...,\mu^* \omega$ with $\mu^* \omega$ appearing $n-m$ times and noting that $\nd(\mu^* \omega)=n$, we get that
\begin{equation*}
  \mu^*\alpha_1\cdot...\cdot \mu^*\alpha_m \cdot \mu^*\omega^{n-m}> 0,
\end{equation*}
which implies the desired result $\alpha_1\cdot...\cdot\alpha_m \cdot \omega^{n-m}> 0$.

This finishes the proof.

\end{proof}

\section{Polymatroid structures}\label{sec polymatroid}

In this section, we discuss various combinatorial structures closely related to Theorem \ref{hardlef tori} and Theorem \ref{intro thrm interPosi}.

Our starting point is the application to multidegrees in the K\"ahler setting.

In \cite{CCLMZ20}, the authors systematically explored when the multidegrees of an algebraic subvariety $X$ (over an arbitrary field $k$) of the product of projective spaces $$\mathbb{P}=\mathbb{P}_1 \times ...\times \mathbb{P}_m$$ is positive and studied its algebraic and combinatorial aspects.
To introduce our result, we give a brief introduction to their nice result.
For simplicity, assume that the field $k$ is algebraically closed and $X$ is irreducible, for each $$\mathbf{n}=(n_1,...,n_m)\in \mathbb{N}^m$$ with
$$n_{[m]}=n_1 +...+n_m = \dim X,$$
\emph{the multidegrees of $X$ of type $\mathbf{n}$ with respect to $\mathbb{P}$}, denoted by $\deg_{\mathbb{P}} ^{\mathbf{n}} (X)$, is the number of intersection points (counting multiplicities) of $X$ with the product $$L_1 \times ...\times L_m \subset \mathbb{P},$$
where $L_k \subset \mathbb{P}_k$ is a general linear subspace of codimension $n_k$ for each $1 \leq k\leq m$. Equivalently,
\begin{equation*}
\deg_{\mathbb{P}} ^{\mathbf{n}} (X) = H_1 ^{n_1} \cdot ... \cdot H_m ^{n_m} \cdot [X],
\end{equation*}
where $H_k$ is the hyperplane divisor class on $\mathbb{P}_k$ for each $1 \leq k\leq m$.

Given a subset $I\subset [m]$, denote by $\pi_I$ the projection
\begin{equation*}
\pi_I:  \mathbb{P}=\mathbb{P}_1 \times ...\times \mathbb{P}_m \rightarrow \mathbb{P}_{i_1} \times ...\times \mathbb{P}_{i_{|I|}}.
\end{equation*}
The the positivity of $\deg_{\mathbb{P}} ^{\mathbf{n}} (X)$ is determined by the dimensions of $\pi_I (X)$, more precisely, the following statements are equivalent:
\begin{itemize}
  \item $\deg_{\mathbb{P}} ^{\mathbf{n}} (X) >0$;
  \item $\dim \pi_I (X) \geq n_I, \forall I \subset [m]$.
\end{itemize}

When $m=2$, the positivity of $\deg_{\mathbb{P}} ^{\mathbf{n}} (X)$ was considered in \cite{trungposimixed, huhgraph}.

It is easy to generalize multidegrees to the setting when the total space $\mathbb{P}$ is replaced by other product spaces.

\begin{defn}\label{defn kahler multidegree}
Let $X$ be a K\"ahler submanifold of $Y= (Y_1, \omega_1) \times ...\times (Y_m, \omega_m)$, where each $Y_k$, $1\leq k\leq m$, is a compact K\"ahler manifold with a fixed K\"ahler class $\omega_k$. Then for each $$\mathbf{n}=(n_1,...,n_m)\in \mathbb{N}^m$$ with
$$n_{[m]}=n_1 +...+n_m = \dim X,$$
\emph{the multidegrees of $X$ of type $\mathbf{n}$ with respect to $Y$}, denoted by $\deg_{Y} ^{\mathbf{n}} (X)$, is defined by
$$\deg_{Y} ^{\mathbf{n}} (X) = \omega_1^{n_1}\cdot ... \cdot \omega_m^{n_m}\cdot [X]. $$
\end{defn}

For $I\subset [m]$, denote by $\pi_I$ the projection
\begin{equation*}
\pi_I:  Y=Y_1 \times ...\times Y_m \rightarrow Y_{i_1} \times ...\times Y_{i_{|I|}}.
\end{equation*}

As an application of Theorem \ref{thrm interPosi}, we give a sufficient and necessary condition for the fundamental question when $\deg_{Y} ^{\mathbf{n}} (X)>0$ in the K\"ahler setting.

\begin{thrm}\label{multidegree}
Setting as in the above definition, then the following statements are equivalent:
\begin{itemize}
  \item $\deg_{Y} ^{\mathbf{n}} (X)>0$;
  \item $\dim \pi_I(X)\geq n_I$ for any $I \subset [m]$.
\end{itemize}

\end{thrm}

\begin{proof}
Denote by $\alpha_i=(\pi_i^*\omega_i)_{|X}$, then
  \begin{equation*}
    \deg_Y ^{\mathbf{n}} (X) =\int_X \alpha_1^{n_1}\cdot ... \cdot \alpha_m^{n_m}.
  \end{equation*}

Write $$\widetilde{\alpha}_I=n_{i_1} \alpha_{i_1}+... +n_{i_{|I|}}\alpha_{i_{|I|}}$$ for $I\subset [m]$, applying Theorem \ref{thrm interPosi} to the nef classes $$\alpha_1 ,...,\alpha_1, ...., \alpha_m ,...,\alpha_m$$ with each $\alpha_k$ repeating $n_k$ times,
it is clear that the intersection number is positive if and only if
 \begin{align*}
   \nd(\widetilde{\alpha}_I)&=\nd(n_{i_1} \alpha_{i_1}+... +n_{i_{|I|}}\alpha_{i_{|I|}})\\
   &\geq n_I
 \end{align*}
  for any $I\subset [m]$.

We claim that $\nd(\widetilde{\alpha}_I)=\nd(\alpha_I)=\dim \pi_I(X)$, where the first equality is obvious.

To this end, without loss of generality, we may assume that $X\subset Y_1\times Y_2$.
We show that $$\nd(\alpha_1)=\dim\pi_1(X).$$

Suppose $\dim\pi_1(X)=d_1$. Let $A$ be the analytic subset of $X$ where $$\text{rank} (d{\pi_1}_{|X})<d_1.$$
Then $\pi_1(A)$ is an analytic subset of $Y_1$ by Remmert's proper mapping theorem. Moreover, we have
  \begin{align*}
    \dim \pi_1(A) &= \max\{ \text{rank}(d\pi_1)_x | x\in A_{\text{reg}}   \}\\
    & < d_1.
  \end{align*}
Hence $\pi_1(A)$ is a proper subset of $\pi_1(X)$. Let $B=\pi_1(A)\cup \pi_1(X)_{\text{sing}}$. Then
  \begin{equation*}
    \pi_1: X\backslash \pi_1^{-1}(B) \rightarrow \pi_1(X)\backslash B
  \end{equation*}
is a submersion. It is clear that for any $y\in X\backslash \pi_1^{-1}(B)$,
$$\alpha_1^{d_1}(y)> 0,\ \text{and}\ \alpha_1^{l}(y)= 0$$
for any $l>d_1$.
Since $X\backslash \pi_1^{-1}(B)$ is a dense open set of $X$, together with the semi-positivity of $\alpha_1$, we find that
  \begin{equation*}
    \nd(\alpha_1)=d_1=\dim\pi_1(X).
  \end{equation*}

This finishes the proof.
\end{proof}

\begin{rmk}
By Corollary \ref{kahler variety interspos}, Theorem \ref{multidegree} extends to compact K\"ahler varieties.
\end{rmk}

It was shown in \cite{CCLMZ20} that in the algebraic case, $r(I):=\dim \pi_I (X)$ is a submodular function by interpreting $\dim \pi_I (X)$ as the transcendental degree of an extension field.

Combining with our geometric observation that $\dim \pi_I (X) = \nd (\widetilde{\alpha}_I)$, it motivate us to consider if the numerical dimension is a submodular function. We prove the following result:

\begin{thrm}\label{rank}
Let $X$ be a complex projective manifold of dimension $n$, then for any three nef classes $A,B,C \in H^{1,1} (X, \mathbb{R})$ on $X$, we always have
\begin{equation*}
  \nd(A+B+C) + \nd(C)\leq \nd(A+C)+\nd(B+C).
\end{equation*}
\end{thrm}

\begin{proof}

In order to prove the inequality, it is sufficient to verify the following statement:
\begin{quote}
  Let $k, l, m$ be non-negative integers such that
   \begin{align*}
     &(A+C)^{k+1}=0,\\
     &(B+C)^{l+1}=0,\\
     &C^m \neq 0, C^{m+1}=0,
   \end{align*}
then $(A+B+C)^{k+l-m+1}=0$.
\end{quote}

To this end, it is equivalent to prove that for any triple of nonnegative integers $(s_1, s_2, s_3)$ satisfying $$s_1 + s_2 + s_3 = k+l-m+1,$$
we have that $$A^{s_1} \cdot B^{s_2} \cdot C^{s_3}=0.$$

Fix a very ample class $\omega$, and let $V$ be an irreducible subvariety in the cycle class $\omega^{n-(k+l-m+1)}$.

For $\varepsilon>0$, set $D=\varepsilon \omega+ C$.

By Lemma \ref{reverse KT}, we have that

\begin{align*}
 & A^{s_1} \cdot B^{s_2} \cdot C^{s_3} \cdot [V]\\
 &  \leq \frac{(k+l-m+1)!}{s_2 ! (k+l-m+1-s_2)!} \frac{(A^{s_1}  \cdot C^{s_3}\cdot D^{s_2} \cdot [V])(B^{s_2} \cdot D^{k+l-m+1-s_2} \cdot [V])}{D^{k+l-m+1} \cdot [V]}.
\end{align*}

For notational simplicity, in the following we use the notations $\simeq, \preceq$: for two real numbers $a, b$, $a\simeq b$ if there are two constants $c_1, c_2 >0$ such that $c_1 a \leq b \leq c_2 a$; $a \preceq b$ if there is a constant $c>0$ such that $a \leq c b$.

By the assumption $C^m \neq 0, C^{m+1}=0$,
\begin{equation*}
  D^{k+l-m+1} \cdot [V] \simeq \varepsilon^{k+l-2m+1} + ...
\end{equation*}
where $...$ is a term with higher power than $\varepsilon^{k+l-2m+1}$. Similarly, by the assumptions $$(A+C)^{k+1}=0,(B+C)^{l+1}=0,$$
we get that
\begin{align*}
   & A^{s_1}  \cdot C^{s_3}\cdot D^{s_2} \cdot [V] \preceq  \varepsilon^{l-m+1} + ...,\\
   & B^{s_2} \cdot D^{k+l-m+1-s_2} \cdot [V] \preceq \varepsilon^{k-m+1} + ....
\end{align*}

Putting the above estimates together implies that
\begin{equation*}
  A^{s_1} \cdot B^{s_2} \cdot C^{s_3} \cdot [V] \preceq \frac{\varepsilon^{k+l-2m+2}+...}{\varepsilon^{k+l-2m+1}+...},
\end{equation*}
which tends to zero as $\varepsilon\rightarrow 0$.
Therefore, $$A^{s_1} \cdot B^{s_2} \cdot C^{s_3} =0.$$

This finishes the proof.
\end{proof}

\begin{rmk}
We expect that Theorem \ref{rank} also holds for nef classes on an arbitrary compact K\"ahler manifold, this depends on a relative version of the reverse Khovanskii-Teissier inequalities by replacing the cycle class $[V]$ by a general complete intersection of transcendental K\"ahler classes.
\end{rmk}

\begin{rmk} \label{nd subadd}
By the definition of numerical dimension, for any two nef $(1,1)$ classes $A, B$, it is clear that
\begin{equation*}
  \nd(A+B)\leq \nd(A) +\nd(B).
\end{equation*}
This is a special case of Theorem \ref{rank} when $C=0$.

\end{rmk}

\begin{rmk} \label{nd equi estimate}
Notations as Theorem \ref{rank}, by the definition of numerical dimension we have that
\begin{equation*}
  \nd(A+B+C) \geq \max\{\nd(A+C), \nd(B+C)\},
\end{equation*}
which yields that
\begin{equation*}
  2(\nd(A+B+C)+\nd(C)) \geq \nd(A+C)+ \nd(B+C).
\end{equation*}
Therefore combining with Theorem \ref{rank},
\begin{equation*}
\frac{1}{2} (\nd(A+C)+ \nd(B+C))\leq  \nd(A+B+C)+\nd(C) \leq \nd(A+C)+ \nd(B+C).
\end{equation*}
\end{rmk}

\begin{cor} \label{polymatroid2}
Given any finite nonzero nef classes $B_1, ..., B_m$ on a complex projective manifold, which we identify the collection with the indices set $E=[m]$. For any $I\subset [m]$, set $r(I)=\nd(B_I)$ with the convention that $r(\emptyset)=0$. Then the function $r(\cdot)$ endows with $E$ a loopless polymatroid structure.

\end{cor}

\begin{proof}
For any nonempty $I, J\subset E$, it is clear that $$\nd(B_I + B_J+ B_{I \cap J}) = \nd (B_{I\cup J}).$$ Applying Theorem \ref{rank} to $B_I, B_J, B_{I \cap J}$ yields the submodularity:
\begin{equation*}
r(I\cup J) + r(I \cap J) \leq r(I)+ r(J).
\end{equation*}

It is obvious that for any $I\subset J \subset E$, $r(I) \leq r(J)$ -- this verifies the monotonicity.

For any nonempty $I\subset E$, $r(I)\neq 1$, which verifies the loopless.

This finishes the proof.

\end{proof}

Combining Theorem \ref{rank} with Corollary \ref{polymatroid2},
we finish the proof of Theorem \ref{polymatroid}.

\begin{rmk}
By Remark \ref{jiangli reverseKT}, it is clear that the same argument works on any smooth projective variety over an algebraically closed field, therefore the analogous result also holds in this setting.
\end{rmk}

Next we give an alternative approach to the submodularity of numerical dimension for semi-positive classes on an arbitrary compact K\"ahler manifold. This follows from the characterization of positivity of complete intersection of nef classes given by
Theorem \ref{intro thrm interPosi}.

\begin{prop}\label{rank on Kahler}
Let $X$ be a compact K\"ahler manifold of dimension $n$ and $A,B,C \in H^{1,1}(X,\mathbb{R})$ be nef classes with semi-positive representatives. Then we have that
  \begin{equation*}
    \nd(A+B+C) + \nd(C)\leq \nd(A+C)+\nd(B+C).
  \end{equation*}
\end{prop}

\begin{proof}
We may fix a semipositive representative in each class, denoted by $\widehat{A},\widehat{B},\widehat{C}$.
Fix a K\"ahler class $\omega$ on $X$, and denote a K\"ahler metric in $\omega$ by $\widehat{\omega}$.

By Theorem \ref{intro thrm interPosi},
  \begin{align*}
    &\int_{X}\widehat{C}^{\nd(C)}\wedge (\widehat{A}+\widehat{B})^{\nd(A+B+C)-\nd(C)} \wedge \widehat{\omega}^{n-\nd(\widehat{A}+\widehat{B}+\widehat{C})}\\
    &=C^{\nd(C)}\cdot (A+B)^{\nd(A+B+C)-\nd(C)} \cdot \omega^{n-\nd(A+B+C)}\\
    &>0.
  \end{align*}
By semi-positivity of the volume form
  \begin{equation*}
    \widehat{C}^{\nd(C)}\wedge (\widehat{A}+\widehat{B})^{\nd(A+B+C)-\nd(C)} \wedge \widehat{\omega}^{n-\nd(A+B+C)}
  \end{equation*}
at any point and positivity of the integration, there is some point $x$ such that
  \begin{align*}
    &\widehat{C}(x)^{\nd(C)}\wedge (\widehat{A}(x)+\widehat{B}(x))^{\nd(A+B+C)-\nd(C)} \wedge \widehat{\omega}(x)^{n-\nd(A+B+C)}\\
    &>0.
  \end{align*}

Then by Corollary \ref{interPos linear} or Theorem \ref{intro thrm interPosi},
\begin{itemize}
  \item $\widehat{C}(x)$ is $\nd(C)$ positive;
  \item $\widehat{A}(x)+\widehat{B}(x)+\widehat{C}(x)$ is $\nd(A+B+C)$ positive.
\end{itemize}

On the other hand, we claim that $\widehat{C}(x)$ is exactly $\nd(C)$ positive and $\widehat{A}(x)+\widehat{B}(x)+\widehat{C}(x)$ is exactly $\nd(A+B+C)$ positive.
If $\widehat{C}(x)$ is not exactly $\nd(C)$ positive, then we will have that
  \begin{equation*}
    \widehat{C}(x)^{\nd(C)+1}\wedge \widehat{\omega}(x)^{n-\nd(C)-1} >0.
  \end{equation*}
Then it follows that
  \begin{equation*}
    C^{\nd(C)+1}\cdot \omega^{n-\nd(C)-1}=\int_X \widehat{C}^{\nd(C)+1}\wedge \widehat{\omega}^{n-\nd(C)-1} >0
  \end{equation*}
by the semi-positivity of $\widehat{C}$. This contradicts to the definition of $\nd(C)$. By the same way, we get that $\widehat{A}(x)+\widehat{B}(x)+\widehat{C}(x)$ is exactly $\nd(A+B+C)$ positive.
Thus we have that
  \begin{align*}
    &\codim \ker(\widehat{A}(x)+\widehat{B}(x)+\widehat{C}(x)) =\nd(A+B+C), \\
    &\codim \ker(\widehat{C}(x)) = \nd(C).
  \end{align*}

Linear algebraic argument (or applying Theorem \ref{rank} to some abelian variety) implies that
  \begin{align*}
    &\codim \ker(\widehat{A}(x)+\widehat{B}(x)+\widehat{C}(x))+\codim \ker(\widehat{C}(x))\\
    &\leq \codim \ker(\widehat{A}(x)+\widehat{C}(x))+\codim \ker(\widehat{B}(x)+\widehat{C}(x)).
  \end{align*}
Again by semi-positivity, we have that
  \begin{align*}
    &\codim \ker(\widehat{A}(x)+\widehat{C}(x)) \leq \nd(A+C),\\
    &\codim \ker(\widehat{B}(x)+\widehat{C}(x)) \leq \nd(B+C).
  \end{align*}

Take all estimates above into account, we get that
  \begin{equation*}
    \nd(A+B+C) + \nd(C)\leq \nd(A+C)+\nd(B+C).
  \end{equation*}

This finishes the proof.
\end{proof}

Just as Theorem \ref{polymatroid}, the numerical dimension endows a finite set of semi-positive classes with a polymatroid structure.

\begin{cor}\label{polymatroidKahler}
  Let $X$ be a a compact K\"ahler manifold and let $\alpha_1,..., \alpha_m \in H^{1,1}(X,\mathbb{R})$ be nonzero semi-positive classes on $X$. Then
  \begin{equation*}
    r(I)=\nd(\alpha_I), \forall I \subset [m]
  \end{equation*}
  endows $[m]$ with a loopless polymatroid structure.
\end{cor}

\begin{proof}
  Apply Proposition \ref{rank on Kahler} and the same argument as Corollary \ref{polymatroid2}.
\end{proof}

\begin{rmk}
It is interesting to see whether the above polymatroids given by nef or semi-positive classes are algebraic (or even representable), see Example \ref{matroid exmpl}.

In the case when $\alpha_1,...,\alpha_m$ are given by semi-ample line bundles on a complex projective manifold $X$, the polymatroid is representable. We can assume that every $\alpha_k$ is given by a free line bundle $A_k$. Let $$\phi_k  : X \rightarrow \mathbb{P}_k = \mathbb{P}(H^0 (X, A_k))$$
be the Kodaira map defined by $A_k$, then we have a morphism
$$\phi= \phi_1 \times...\times \phi_m : X \rightarrow \mathbb{P}=\mathbb{P}_1 \times...\times \mathbb{P}_m.$$
Denote $Y=\phi (X) \subset \mathbb{P}$, then we have that $\dim \pi_I (Y) = \kappa (A_I) =\nd(A_I)$. By \cite{CCLMZ20}, $r(I)=\dim \pi_I (Y)$ endows the set $\alpha_1,...,\alpha_m$ with an algebraic polymatroid structure. Since over a field of characteristic zero any algebraic polymatroid is representable, the desired result holds. One such construction for matroids also appeared in \cite{huhwangEnumeration} where the projective manifold is a wonderful compactification and can be further generalized to polymatroids \cite{huhwangPolymatroid}.

However, in general one cannot reduce the nef case to the semi-ample case, since there exist nef $(1,1)$ classes which cannot be made to be semi-ample by any modifications.
\end{rmk}

Let $X\subset Y = Y_1 \times ...\times Y_m$ be in the setting of Definition \ref{defn kahler multidegree}, define \emph{the support of $X$ with with respect to $Y$} by
$$\MSupp _Y (X) = \{\mathbf{n}\in \mathbb{N}^m| \deg^{\mathbf{n}} _Y (X) >0\}.$$

\begin{cor}\label{support}
The support of $X$ with with respect to $Y$, $\MSupp _Y (X)$, is a discrete polymatroid.
\end{cor}

\begin{proof}
By Theorem \ref{multidegree}, we have that
  \begin{equation*}
    \MSupp _Y (X)=\{ \mathbf{n}\in \mathbb{N}^m | n_{[m]}=\dim X=n, \dim \pi_I(X)\geq n_I , \forall I \subsetneq [m] \}.
  \end{equation*}
  As we see in the proof of Theorem \ref{multidegree},
  \begin{equation*}
    \nd(\alpha_I)=\dim \pi_I(X).
  \end{equation*}

Here the $\alpha_I$ are semi-positive classes on $X$. Then by Proposition \ref{rank on Kahler},
  $$r(I)=\nd(\alpha_I)=\dim \pi_I(X)$$ is a rank function. It follows that $\MSupp _Y (X)$ is a discrete polymatroid with rank function $r(I)$ by the submodular theorem (see \cite{Submodular}).
\end{proof}

This extends related result of \cite{CCLMZ20} to the K\"ahler setting.
For a projective variety over an algebraically closed field, \cite{branhuhlorentz} also proved the result using the theory of Lorentzian polynomials.


Finally, in the linear setting we show that discrete polymatroids appear naturally in the HL property.

\begin{prop}\label{polym HL linear}
Let $X=\mathbb{C}^n / \Gamma$ be a compact complex torus of dimension $n$, and let $\alpha_1,...,\alpha_m$ be nef $(1,1)$ classes on $X$. Denote
\begin{equation*}
  M=\{\mathbf{n}=(n_1,...,n_m)\in \mathbb{N}^m | \sum_{i=1} ^m n_i =m,\ \alpha_1 ^{n_1} \cdot...\cdot \alpha_m ^{n_m}\ \text{has HL property.}\},
\end{equation*}
then $M$ is a discrete polymatroid with the rank function $r(I) = \nd(\alpha_I) - (n-m)$.
\end{prop}

\begin{proof}
By Theorem \ref{hardlef tori}, if $\alpha_1 ^{n_1} \cdot...\cdot \alpha_m ^{n_m}$ has HL property, then we must have $$\nd(\alpha_k) \geq n-m+1.$$
Therefore, for any $I\subset [m]$, $$r(I) = \nd(\alpha_I) - (n-m) \in \mathbb{N}.$$

By Theorem \ref{hardlef tori} again, we have that
\begin{equation*}
  M=\{\mathbf{n}=(n_1,...,n_m)\in \mathbb{N}^m | \sum_{i=1} ^m n_i =m,\ r(I) \geq n_I, \forall I\subsetneq [m]\}.
\end{equation*}
On the other hand, by Corollary \ref{polymatroidKahler}, $r(\cdot)$ is a rank function on $[m]$.

Thus, $M$ is a discrete polymatroid with the rank function $r(I) = \nd(\alpha_I) - (n-m)$.

\end{proof}


\bibliography{reference}
\bibliographystyle{amsalpha}

\bigskip

\bigskip

\noindent
\textsc{Tsinghua University, Beijing 100084, China}\\
\noindent
\verb"Email: hujj22@mails.tsinghua.edu.cn"\\
\noindent
\verb"Email: jianxiao@tsinghua.edu.cn"

\end{document}